\documentclass[10pt]{article}

\usepackage{graphicx,subfigure}
\usepackage{amsmath,amsthm,amssymb,enumerate}
\usepackage{euscript,mathrsfs}
\usepackage{color}
\usepackage{dsfont}
\usepackage{url}
\usepackage{bm}
\usepackage[left=2cm,right=2cm,top=3.5cm,bottom=3.5cm]{geometry}
\usepackage[framemethod=tikz]{mdframed}
\usepackage{soul}
\usepackage[colorlinks=true,       
linkcolor=black,       
citecolor=blue,      
urlcolor=magenta,     
bookmarks=true,       
pdfstartview=FitH]{hyperref}
\usepackage{cleveref}

\catcode`\@=11 \@addtoreset{equation}{section}

\catcode`\@=12

\theoremstyle{definition}
\newtheorem{Theorem}{Theorem}[section]
\newtheorem{Proposition}[Theorem]{Proposition}
\newtheorem{Lemma}[Theorem]{Lemma}
\newtheorem{Corollary}[Theorem]{Corollary}
\newtheorem{Definition}[Theorem]{Definition}

\newtheorem{Remark}[Theorem]{Remark}

\newcommand{\bTheorem}[1]{\begin{Theorem} \label{T#1}}
	\newcommand{\eT}{\end{Theorem}}

\newcommand{\bProposition}[1]{\begin{Proposition} \label{P#1}}
	\newcommand{\eP}{\end{Proposition}}

\newcommand{\bLemma}[1]{\begin{Lemma} \label{L#1}}
	\newcommand{\eL}{\end{Lemma}}

\newcommand{\bCorollary}[1]{\begin{Corollary} \label{C#1}}
	\newcommand{\eC}{\end{Corollary}}

\newcommand{\bRemark}[1]{\begin{Remark} \label{R#1}}
	\newcommand{\eR}{\end{Remark}}

\newcommand{\bDefinition}[1]{\begin{Definition} \label{D#1}}
	\newcommand{\eD}{\end{Definition}}

\newcommand{\bFormula}[1]{\begin{equation} \label{#1}}
	\newcommand{\eF}{\end{equation}}

\definecolor{Cgrey}{rgb}{0.85,0.85,0.85}
\definecolor{Cblue}{rgb}{0.50,0.85,0.85}
\definecolor{Cred}{rgb}{1,0,0}
\definecolor{fancy}{rgb}{0.10,0.85,0.10}
\definecolor{amaranth}{rgb}{0.9, 0.17, 0.31}

\newcommand\Cbox[2]{%
	\newbox\contentbox%
	\newbox\bkgdbox%
	\setbox\contentbox\hbox to \hsize{%
		\vtop{
			\kern\columnsep
			\hbox to \hsize{%
				\kern\columnsep%
				\advance\hsize by -2\columnsep%
				\setlength{\textwidth}{\hsize}%
				\vbox{
					\parskip=\baselineskip
					\parindent=0bp
					#2
				}%
				\kern\columnsep%
			}%
			\kern\columnsep%
		}%
	}%
	\setbox\bkgdbox\vbox{
		\color{#1}
		\hrule width  \wd\contentbox %
		height \ht\contentbox %
		depth  \dp\contentbox
		\color{black}
	}%
	\wd\bkgdbox=0bp%
	\vbox{\hbox to \hsize{\box\bkgdbox\box\contentbox}}%
	\vskip\baselineskip%
}

\mdfdefinestyle{MyFrame}{%
	linecolor=black,
	outerlinewidth=1pt,
	roundcorner=5pt,
	innertopmargin=\baselineskip,
	innerbottommargin=\baselineskip,
	innerrightmargin=10pt,
	innerleftmargin=10pt,
	backgroundcolor=white!20!white}

\allowdisplaybreaks

\newcommand{\R}{\mathbb{R}}

\newcommand{\veps}{\varepsilon}
\newcommand{\eps}{\veps}

\newcommand{\lam}{\lambda}

\renewcommand{\div}{{\rm div}\,}

\newcommand{\vu}{\bm{u}}

\newcommand{\vm}{\bm{m}}

\newcommand{\vn}{\bm{n}}

\newcommand{\vU}{\bm{U}}

\newcommand{\vF}{\bm{F}}

\newcommand{\vx}{\bm{x}}
\newcommand{\vV}{\bm{V}}
\newcommand{\vS}{\bm{S}}

\newcommand{\tvU}{\widetilde{\vU}}

\newcommand{\tM}{\widetilde{M}}

\newcommand{\tp}{\widetilde{p}}
\newcommand{\tu}{\widetilde{u}}

\newcommand{\cV}{\mathcal{V}}

\newcommand{\cD}{\mathcal{D}}

\newcommand{\aleq}{\stackrel{<}{\sim}}

\newcommand{\oline}{\overline}

\def\softd{{\leavevmode\setbox1=\hbox{d}%
		\hbox to 1.05\wd1{d\kern-0.4ex{\char039}\hss}}}

\begin{document}


\title{\bf On admissible solutions to the coupled Riemann problem with heat-flux discontinuity}

\author{
	Changsheng Yu
	\thanks{The work was  supported by the Deutsche Forschungsgemeinschaft (DFG, German Research Foundation) - Project number 525853336 -- SPP 2410 ``Hyperbolic Balance Laws: Complexity, Scales and Randomness.''}
	\and
    Tiegang Liu$^\dagger$
    }

\date{}

\maketitle

\centerline{$^*$Institute of Mathematics, Johannes Gutenberg-University Mainz}
\centerline{Staudingerweg 9, 55 128 Mainz, Germany}
\centerline{chayu@uni-mainz.de}

\medskip

\centerline{$^\dagger$School of Mathematical Sciences, Beihang University}
\centerline{Beijing, 100191, PR China}
\centerline{liutg@buaa.edu.cn}


\begin{abstract}
	We study the Riemann problem for the compressible Euler equations with a stationary coupling interface across which a discontinuity in the heat flux is prescribed. This coupling gives rise to non-conservative effects and models heat addition mechanisms such as condensation-induced waves. Without imposing restrictions on sonic states, we analyze the problem in all Mach number regimes.
	
	Lax weak entropy solutions are constructed via half-Riemann problems, and we show that non-uniqueness occurs for a large class of initial data. To address this, we introduce an admissibility criterion derived from the evolutionarity criterion, and we characterize the full structure of admissible Riemann solutions. Our analysis establishes local existence of admissible Riemann solutions provided the heat flux jump is sufficiently small, while also identifying families of initial data for which admissible Riemann solutions cannot exist for any fixed, nonzero heat flux jump. Numerical experiments are included to illustrate the theoretical findings.
\end{abstract}

{\small

\medbreak

}


\section{Introduction}

Coupled systems of conservation laws arise naturally in many areas of fluid dynamics and
engineering whenever different flow regions, models, or physical effects are connected
through an interface. Typical examples include gas dynamics in networks of pipes and
junctions \cite{holden2013network,d2010modeling,MR3200227}, multiphase flows with phase transition
across moving interfaces \cite{abeyaratne1990driving}, shallow water flows coupled with
topography or man-made structures such as dams and sluices \cite{han2014exact,ungarish2011two},
traffic flow on road networks \cite{garavello2006traffic,MR2209753}, and blood flow in arterial
networks \cite{quarteroni2001mathematical}. In these applications, each branch or subdomain
is modeled by a system of hyperbolic conservation laws, while coupling conditions at the
interfaces ensure mass, momentum, or energy balance. 

\begin{figure}[!h]
	\centering
	\includegraphics[width=0.4\linewidth]{ 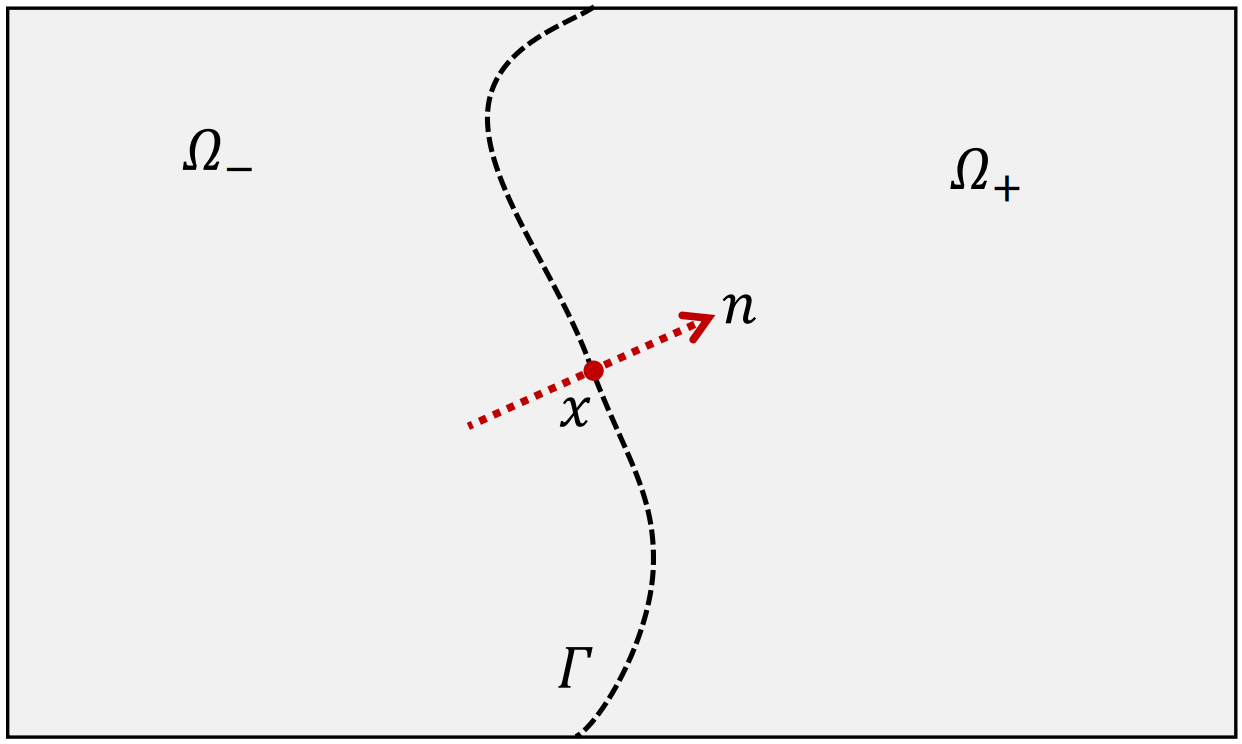}
	\caption{The coupled problem of the compressible Euler system on the $\Omega$.}
	\label{fig7}
\end{figure}
This paper investigates the coupled problem for the compressible fluid flow. 
The domain $\Omega \subset \R^d$ (see Figure~\ref{fig7}) is decomposed 
into two subdomains $\Omega_-$ and $\Omega_+$ separated by a stationary, smooth 
coupling interface $\Gamma$ with a unit normal vector $\vn$. 
In each subdomain, the fluid motion is governed by the compressible Euler system
\begin{equation*}
	\frac{\partial \vU}{\partial t} + \nabla \cdot F(\vU) = 0,
\end{equation*}
where the vector of conservative variables and the flux are given by
\begin{equation*}
	\vU =
	\begin{pmatrix}
		\rho \\ \vm \\ E
	\end{pmatrix}, 
	\qquad 
	F(\vU) =
	\begin{pmatrix}
		\rho \vu \\[0.3em] \rho \vu \otimes \vu + p I_d \\[0.3em] (E+p)\vu
	\end{pmatrix}.
\end{equation*}
Here, $\rho$ denotes the density, $\vu \in \R^d$ the velocity, 
$\vm = \rho \vu$ the momentum, and $E$ the total energy. 
The pressure $p$ is related to the internal energy $e$ via a polytropic equation of state (EOS),
\begin{equation*}
	p = (\gamma - 1)\rho e, \qquad \gamma > 1,
\end{equation*}
with
\begin{equation*}
	e = \frac{1}{\rho} \left( E - \tfrac{1}{2}\rho |\vu|^2 \right).
\end{equation*}

On the coupling interface $\Gamma$, we impose a precondition. 
The traces of the solution in the two subdomains are defined as
\begin{equation*}
	\begin{aligned}
		&\vx^-=\lim_{\epsilon\to 0^-} \vx+\epsilon \vn ,\quad
		\vx^+=\lim_{\epsilon\to 0^+} \vx+\epsilon \vn ,\\
		&\vU_-=\lim_{\epsilon\to 0^-} \vU(t,\vx+\epsilon \vn ),\quad
		\vU_+=\lim_{\epsilon\to 0^+} \vU(t,\vx+\epsilon \vn ).
	\end{aligned}
\end{equation*}
Here, $\vn$ denotes the unit normal vector at $\vx \in \Gamma$.  
In general, the coupling condition can be characterized by a nonlinear relation
\[
\Psi(\vU_-,\vU_+)=0 .
\]

In this work, we focus on the so-called \emph{heat-flux jump} condition, which reads
\begin{equation*}
	\Psi(\vU_-,\vU_+)=
	\begin{pmatrix}
		\rho_-\vu_-\cdot \vn-\rho_+\vu_+\cdot \vn\\[0.3em]
		\rho_-|\vu_-|^2+p_--\rho_+|\vu_+|^2-p_+ \\[0.3em]
		(1+k)(E_-+p_-)\vu_-\cdot\vn-(E_++p_+)\vu_+\cdot \vn
	\end{pmatrix},
\end{equation*}
where $k>0$ is a fixed parameter.

Note that if $\vu\cdot \vn=0$, the condition degenerates into the simple continuity relation $\vU_-=\vU_+$.  
To exclude this trivial case, we assume throughout that $\vu_-\cdot \vn>0$, which implies $\vu_+\cdot \vn>0$ from $\Psi(\vU_-,\vU_+)=0$.
Unlike many previous studies, we do not impose any restrictions on sonic states. 
In particular, our analysis is valid across all Mach number regimes, from subsonic to supersonic, thereby providing a unified framework.

This model has been proposed in the literature to describe flows with condensation 
\cite{cheng2010condensation,schnerr2005unsteadiness}.  
Such flows may exhibit complex wave patterns, as observed both experimentally and numerically, see e.g.~\cite{MR4413539,horie2007shockwave}.  
These structures can be validated through the following Rayleigh-type coupling relations:
\begin{equation*}
	\begin{aligned}
		\rho_-\vu_- &= \rho_+\vu_+,\\
		\rho_-|\vu_-|^2+p_- &= \rho_+|\vu_+|^2+p_+,\\
		(1+k)\left(h_-+\tfrac{|\vu_-|^2}{2}\right) &= h_+ + \tfrac{|\vu_+|^2}{2},
	\end{aligned}
\end{equation*}
where $h=(e+p)/\rho$ denotes the specific enthalpy.  
This heat-flux jump condition is precisely equivalent to $\Psi(\vU_-,\vU_+)=0$.

We consider the Cauchy problem with initial data
\begin{equation*}
	\vU(0,x)=\vU_0(x).
\end{equation*}
In the spirit of half-Riemann problems, such a setting can be viewed as two separate problems posed on the subdomains $\Omega_-$ and $\Omega_+$:
\begin{equation}\label{E1}
	\begin{aligned}
		&\frac{\partial \vU}{\partial t}+\div_x F(\vU)=0,\qquad \vU:[0,T]\times \Omega_- \to \cD \subset \R^{d+2},\\
		&\vU(0,x)=\vU_{0}(x),\quad x\in \Omega_-,
	\end{aligned}
\end{equation}
and
\begin{equation}\label{E2}
	\begin{aligned}
		&\frac{\partial \vU}{\partial t}+\div_x F(\vU)=0,\qquad \vU:[0,T]\times \Omega_+ \to \cD \subset \R^{d+2},\\
		&\vU(0,x)=\vU_{0}(x),\quad x\in \Omega_+.
	\end{aligned}
\end{equation}

Following \cite{herty2019coupling}, we now introduce the notion of weak solutions for the coupled problem.
\begin{Definition}\label{def1}
	A function $\vU:[0,T]\times \Omega \to \cD \subset \R^{d+2}$ is called a \emph{weak solution} of the coupled problem if 
	\begin{itemize}
		\item $\vU$ is a weak solution of the systems \eqref{E1} and \eqref{E2}, respectively; 
		\item the traces $\vU_-$ and $\vU_+$ exist for $x\in \Gamma$ and satisfy the coupling condition
		\begin{equation*}
			\Psi(\vU_-,\vU_+)=0 \qquad \text{for a.e. } t\in [0,T].
		\end{equation*}
	\end{itemize}
\end{Definition}

We focus on the Riemann problem associated with the one-dimensional coupled system, which is defined by the initial condition
\begin{equation*}
	\vU_0(x)=\begin{cases}
		\vU_L, & x\leq 0,\\
		\vU_R, & x>0,
	\end{cases}
\end{equation*}
where $\vU_L,\vU_R \in \cD$ are two prescribed states.
This setting can be interpreted as the projection of the multi-dimensional coupled problem onto the normal direction of the interface $\Gamma$.
Consequently, the coupling interface reduces to a single point in $\R$, which, without loss of generality, is located at the origin.

The coupling interface acts as a stationary discontinuity in the Riemann solutions. 
This leads to complex wave patterns, especially in the absence of any restriction on the Mach numbers. 
In addition, the interface states may admit multiple branches, giving rise to the non-uniqueness of Riemann solutions for a broad class of initial data. 
To address this issue, we introduce an admissibility criterion inspired by the evolutionarity condition for general discontinuities in fluid dynamics, and we characterize the resulting structure of admissible solutions.

Our main contributions can be summarized as follows. 
First, we establish the local existence of admissible Riemann solutions whenever the heat-flux jump across the interface is sufficiently small. 
Second, we show that for certain families of initial data, no admissible solution exists if the heat-flux discontinuity is fixed and nonzero. 
Finally, we present numerical verifications that not only confirm the predicted solution structures but also reveal the emergence of self-similar yet non-admissible solutions in the absence of monotonicity criterion. 
These results provide a rigorous foundation for the study of coupled hyperbolic systems with non-conservative interfaces and may guide future developments in the analysis of networked and multiphysics flow models.

This paper is structured as follows.
Section~\ref{sec2} establishes the structural properties of coupled Riemann solutions.
Section~\ref{sec3} discusses the non-uniqueness of entropy solutions and motivates the admissibility criterion introduced in Section~\ref{sec4}.
Section~\ref{sec5} proves local existence of admissible solutions for sufficiently small heat-flux jumps, whereas Section~\ref{sec6} shows the existence of initial data without admissible solutions.
Numerical experiments verifying the theoretical findings are presented in Section~\ref{sec7}, and concluding remarks are given in Section~\ref{sec8}.

\section{Coupled Riemann Problem}\label{sec2}

Before addressing the coupled Riemann problem, we recall some standard results for the classical Riemann problem of the Euler system, corresponding to the degenerate coupling interface $k=0$:
\begin{equation*}\left\lbrace 
	\begin{aligned}
		&\frac{\partial \vU}{\partial t}+\frac{\partial \vF(\vU)}{\partial x}=0,\quad  \vU:[0,T]\times \R \to \cD \in \R^{3}\\
		&\vU(0,x)=\vU_0(x)
	\end{aligned}\right. 
\end{equation*}
		
The compressible Euler system is strictly hyperbolic with three characteristic fields
\begin{equation*}
	\lambda_1(\vU)=u-c,\quad \lambda_2(\vU)=u, \quad \lambda_3(\vU)=u+c,
\end{equation*}
and $c=\sqrt{\frac{\gamma p}{\rho}}$ is the speed of sound. 
Solutions of the classical Riemann problem are sought in the self-similar form
\begin{equation*}
	\vV\left( \frac{x}{t}; \vU_L,\vU_R \right) : \R\cup\{\pm \infty\}\times \cD \times \cD \to \cD
\end{equation*}
constructed as a concatenation of three simple waves associated with the characteristic fields. Each simple wave is either a shock wave, a rarefaction wave, or a contact discontinuity. We restrict attention to entropy solutions, in the sense that every discontinuity in the solution satisfies the Lax entropy conditoin. For further details on the construction of entropy solutions to the classical Riemann problem we refer to \cite{toro2013riemann}.
\begin{Definition}[Lax entropy condition]\label{Lax}
	The discontinuity with the left-hand state $\vU_1$ and right-hand state $\vU_2$ is said to satisfy the Lax entropy conditions if there exists an index $i\in\{1,2,3\}$ such that we have either 
	\begin{itemize}
		\item [(i)] 
		\begin{equation*}
			\begin{cases}
				\lam_i(\vU_2)<\sigma<\lam_{i+1}(\vU_2),\\
				\lam_{i-1}(\vU_1)<\sigma<\lam_i(\vU_1),
			\end{cases}
		\end{equation*}
		if the $i$th characteristic field is genuinely nonlinear; or 
		\item [(ii)]
		\begin{equation*}
			\lam_i(\vU_1)=\sigma=\lam_i(\vU_2)
		\end{equation*}
		if the $i$th characteristic field is linearly degenerate.
	\end{itemize}
\end{Definition}

The first and third characteristic fields of the Euler system are genuinely nonlinear, and their associated simple waves are either shock waves or rarefaction waves.
The second characteristic field is linearly degenerate, giving rise to contact discontinuities.
For each characteristic fields, we denote the corresponding Lax curves of simple waves by
\begin{equation*}
	L_i^{\pm}(\eps;\vU_0),\quad i=1,2,3,
\end{equation*}
where the index $i$ specifies the characteristic field, the superscript $\pm$ indicates the forward or backward direction, $\eps$ is a scalar parameter measuring the wave strength, and $\vU_0$ is the reference state from which the curve is issued.
Explicit expressions of these curves are collected in Appendix~\ref{App2}.

We recall the following property of classical Riemann solutions, as established in \cite{herty2019coupling}.
\begin{Proposition}[\bf Solutions of the classical RP]\label{P3}
	Let us assume that the given Riemann data $\vU_L$ and $\vU_R$ satisfy
	\begin{equation}\label{E6}
		u_L+\frac{2c_L}{\gamma-1}>u_R-\frac{2c_R}{\gamma-1}.
	\end{equation}
	Then the classical RP has a unique solution that is not vacuum.
	Moreover, the intermediate states in the solution depend continuously on the initial data.
\end{Proposition}

We now turn to the coupled Riemann problem with a non-degenerate coupling interface.
Following the approach of Herty et al.~\cite{herty2019couplingof}, the coupled Riemann solution is constructed from two half-Riemann problems.
The half-Riemann problem posed in the left subdomain $\Omega_-$ is given by
\begin{equation}\label{E3}
	\begin{aligned}
		&\frac{\partial \vU}{\partial t}+\frac{\partial \vF (\vU)}{\partial x}=0, \ \vU:[0,T]\times \R_- \to \cD \in \R^3, \\
		&\vU(0,x)=\begin{cases}
			\vU_L,\ &x<0\\
			\vU_-,\ &x>0
		\end{cases}
	\end{aligned}
\end{equation}
The half-Riemann problem in the right subdomain $\Omega_+$ is given by
\begin{equation}\label{E4}
	\begin{aligned}
		&\frac{\partial \vU}{\partial t}+\frac{\partial \vF (\vU)}{\partial x}=0, \ \vU:[0,T]\times \R_+ \to \cD \in \R^3, \\
		&\vU(0,x)=\begin{cases}
			\vU_+,\ &x>0\\
			\vU_R,\ &x<0
		\end{cases}
	\end{aligned}
\end{equation}
If both $\vU_-$ and $\vU_+$ are subsonic, then according to \cite{du2025coupled}, they lie within the following sets of boundary states:
\begin{equation*}
	\begin{aligned}
		\overline{\cV}_L(\vU_L)= \left\lbrace
			\vU\in \cD:\ 
			\exists M\in \{ 1,2,3 \},\ \exists (\epsilon_i)_{i=1}^{M}\subset \R \text{ s.t. } \vU=\vU_M \text{ with } \right. \\
			\vU_0=\vU_L,\vU_i:=L_i^+(\epsilon_i;\vU_{i-1}), i=1,...,M \text{ and } \\
			\left.  \lambda_M(\vU)\leq 0 <\lambda_{M+1}(\vU)
		 \right\rbrace 
	\end{aligned}
\end{equation*}
\begin{equation*}
	\begin{aligned}
		\overline{\cV}_R(\vU_R)= \left\lbrace
		\vU\in \cD:\ 
		\exists M\in \{ 1,2,3 \},\ \exists (\epsilon_i)_{i=1}^{M}\subset \R \text{ s.t. } \vU=\vU_M \text{ with } \right. \\
		\vU_0=\vU_R,\vU_i:=L_i^-(\epsilon_i;\vU_{i-1}), i=1,...,M \text{ and } \\
		\left.  \lambda_{M-1}(\vU)\leq 0 <\lambda_{M}(\vU)
		\right\rbrace 
	\end{aligned}
\end{equation*}
Given states $\vU_-\in \overline{\cV}_L(\vU_L)$ and $\vU_+\in \overline{\cV}_R(\vU_R)$, the coupled Riemann solution is constructed by concatenating the solutions of the two half-Riemann problems, subject to the coupling condition $\Psi(\vU_-,\vU_+)=0$.

In our opinion, however, the above sets of boundary states are only suitable in the subsonic regime.
In the supersonic case, even if the boundary states satisfy $\vU_-\in \overline{\cV}_L(\vU_L)$ and $\vU_+\in \overline{\cV}_R(\vU_R)$ together with the coupling condition, they may fail to generate a valid Riemann solution.
As a counterexample, we consider the initial data satisfying
\begin{itemize}
	\item [(i)] $M_L>1$; 
	\item [(ii)] $\vU_-=L_1^+(\eps;\vU_L)\in \overline{\cV}_L(\vU_L)$, $0<\eps<\frac{2\gamma}{\gamma+1}(M_L^2-1)$; 
	\item [(iii)] $\vU_+\in \overline{\cV}_R(\vU_R)$, $\Psi(\vU_-,\vU_+)=0$; 
\end{itemize}
Then the shock wave connecting the left-hand state $\vU_L$ and right-hand state $\vU_-$ has a positive speed.
It is clear that the left half–Riemann problem cannot serve as a coupled Riemann solution on the left half–plane.
This observation indicates that, in the general case where all Mach numbers are involved, the wave speeds in the half–Riemann problems must be explicitly taken into account.

We employ the double CRP framework described in \cite{yu2022riemann2} to construct solutions for the coupled Riemann problem.
Let us consider a self-similar solution \(\vU\) of the coupled Riemann problem with traces \((\vU_-, \vU_+)\). To extend \(\vU\) from the left half-plane to the whole plane, we define the new function: \[ \vV(x,t) = \begin{cases} \vU(x,t), & x<0 \\ \vU_-, & x\geq 0 \\ \end{cases} \] This construction replaces the solution in the right half-plane with the constant state \(\vU_-\), so the extension is constant for \(x \geq 0\).
It is clearly a classical RP solution, therefore contains at most three simple waves corresponding the characteristic fields with negative speeds. This argument also holds similarily for the solution on the right half-plane. This analysis lead to the following modification of the sets of boundary states
\begin{equation*}
	\begin{aligned}
		{\cV}_L(\vU_L):=\big\{ 
		& \vU \in  \overline{\cV}_L(\vU_L) :
		\text{ there exists a classical RP solution } \vV \\ 
		& \text{ such that }  \vV\left(\frac{x}{t}-;\vU_L,\vU \right)\equiv \vU \text{ for all } x\geq 0, t\geq 0
		\big\}  .
	\end{aligned}
\end{equation*}
\begin{equation*}
	\begin{aligned}
		{\cV}_R(\vU_R):=\big\{
		& \vU \in \overline{\cV}_R(\vU_R) :
		\text{ there exists a classical RP solution } \vV \big\} \\ 
		& \text{ such that }  \vV\left(\frac{x}{t}+;\vU,\vU_R \right)\equiv \vU \text{ for all } x\leq 0, t\geq 0
		\big\}  .
	\end{aligned}
\end{equation*}

Then we can demonstrate that all weak solutions of the coupled Riemann problem have boundary states that belong to the modified sets defined above.

\begin{Theorem}[\bf Structure of Coupled Riemann Solutions]\label{T1}
	Given the initial Riemann data \(\vU_L\) and \(\vU_R\), any self-similar, non-vacuum solution \(\vU\) of the coupled Riemann problem must satisfy \(\vU_-\in \mathcal{V}_L(\vU_L)\) and \(\vU_+\in \mathcal{V}_R(\vU_R)\). 
	
	Moreover, for any \(\vU_-\in \mathcal{V}_L(\vU_L)\) and \(\vU_+\in \mathcal{V}_R(\vU_R)\) satisfying the coupling condition \(\Psi(\vU_-,\vU_+)=0\), there exists a unique self-similar, non-vacuum entropy weak solution of the coupled Riemann problem given by
	\begin{equation}\label{E5}
		\vU(x,t)=
		\begin{cases}
			\vV\left( \frac{x}{t}, \vU_L, \vU_- \right), & x < 0, \\
			\vV\left( \frac{x}{t}, \vU_+, \vU_R \right), & x > 0.
		\end{cases}
	\end{equation}
\end{Theorem}

\begin{Remark}
	
	As shown in \cite{yu2024mathematical}, and Remark~\ref{R2} in the following, Definition~\ref{def1} and the theorem built upon it do not cover the case where shock waves interact with the coupling interface. Such situations are commonly referred to as \emph{resonance} in the study of hyperbolic conservation laws with singular or geometric source terms and in nonconservative hyperbolic systems (cf.\cite{MR2097035, gallouet2008resonance}). For the sake of precision, Definition\ref{def1} should therefore be understood as specifying the class of \emph{non-resonant weak solutions}.
	
\end{Remark}

By default, throughout this paper we restrict our attention to non-resonant solutions of the coupled Riemann problem, namely those admitting the canonical form described in Theorem~\ref{T1}. Equivalently, the problem reduces to finding a pair of boundary states $(\vU_-,\vU_+)$ belonging to the admissible sets $\cV_L(\vU_L)$ and $\cV_R(\vU_R)$.

\section{Ill-posedness of Entropy Solutions}\label{sec3}

We now show that the coupled Riemann problem may admit non-unique solutions for a broad class of initial data, even when the Lax entropy condition is imposed. To this end, we present explicitly the ranges of the boundary states $\vU_-$ and $\vU_+$. 

More precisely, given a left-hand boundary state $\vU_-$, we determine all corresponding right-hand boundary states $\vU_+$ that satisfy the coupling condition. The latter, written component-wise, takes the form
\begin{equation*}
	\begin{aligned}
		\rho_-u_-&=\rho_+u_+,\\
		\rho_-u_-^2+p_-&=\rho_+u_+^2+p_+,\\
		(E_-+p_-)u_-(1+k)&=(E_++p_+)u_+.
	\end{aligned}
\end{equation*}
By means of standard manipulations of the Rankine–Hugoniot relations and thermodynamic identities (see, e.g., Anderson~\cite{anderson1990modern}), the state variables $\rho_+$, $u_+$, and $p_+$ can be expressed solely in terms of the right-hand Mach number $M_+$:
\begin{equation}\label{E14}
	\begin{aligned}
		&\frac{\rho_+}{\rho_-}=\left(\frac{M_-}{M_+} \right) ^2\frac{\gamma (M_+)^2+1}{\gamma (M_-)^2+1},\\
		&\frac{u_+}{u_-}=\left(\frac{M_+}{M_-} \right) ^2\frac{\gamma (M_-)^2+1}{\gamma (M_+)^2+1},\\
		&\frac{p_+}{p_-}=\frac{\gamma (M_-)^2+1}{\gamma (M_+)^2+1}.
	\end{aligned}
\end{equation}
The values of the right-hand Mach number $M_+$ have been derived by Cheng et al.~\cite{cheng2010condensation}.
They depend solely on the left-hand Mach number $M_+$, and admit two different branches:
\begin{equation*}
	M_+=\psi_1(M_-,k) \text{ or } M_+=\psi_2(M_-,k),
\end{equation*}
with
\begin{equation*}
	\begin{aligned}
		&\psi_1(M,k):=\sqrt{(1-I)/(1+\gamma I)},\\
		&\psi_2(M,k):=\sqrt{(1+I)(1-\gamma I)},\\
		&I=I(M,k):=\left[ (\gamma M^2+1)^2-(\gamma+1)M^2[(\gamma-1)M^2+2](1+k)\right] ^{\frac{1}{2}}\left( \gamma M^2+1\right) ^{-1}.
	\end{aligned}
\end{equation*}
We illustrate in Figure~\ref{fig1} the behavior of the functions $\psi_1$ and $\psi_2$ for the case $k=0.1$.
\begin{figure}[htbp]
	\centering
	\includegraphics[width=0.4\linewidth]{ 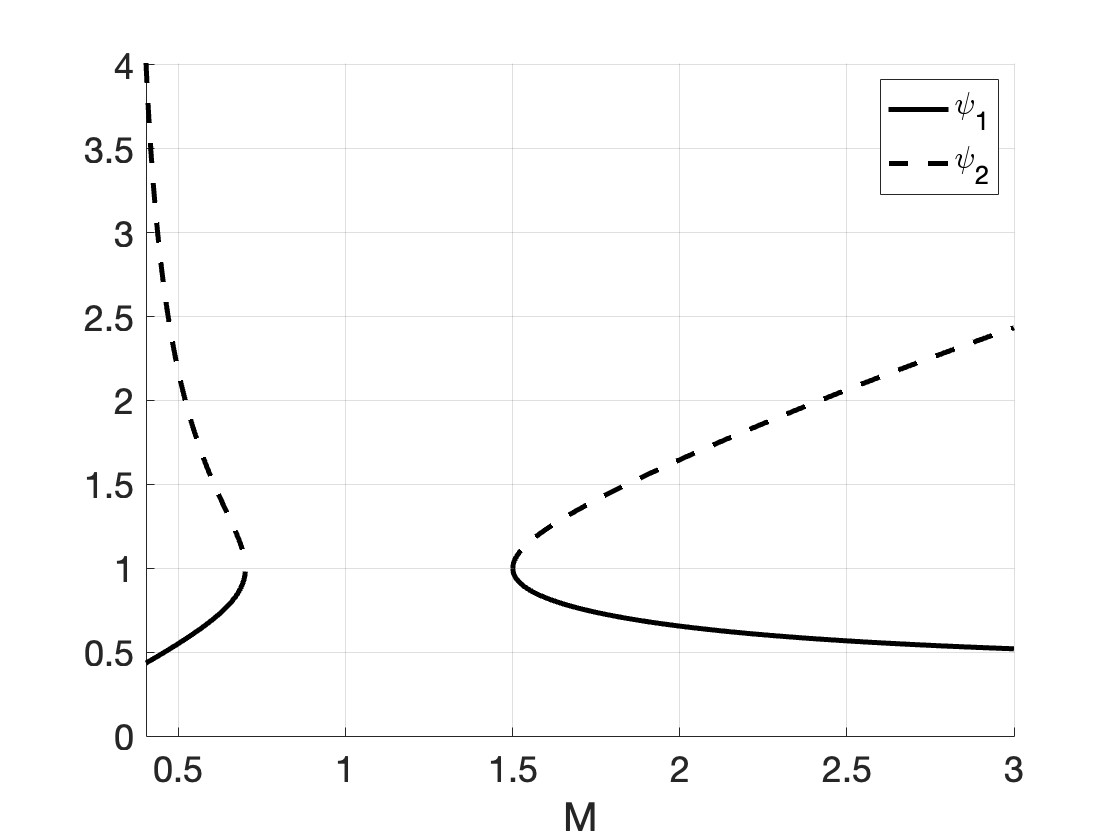}
	\caption{Behaviour of the functions $\psi_1$ and $\psi_2$ with respect to different $M$.}
	\label{fig1}
\end{figure}
It is evident that the Mach numbers $M_-$ and $M_+$ are restricted to certain ranges. To illustrate this, we denote
\begin{equation*}
	\begin{aligned}
	M_1^*&=
	M_1^*(k)=\begin{cases}
		\left[ {\frac{k\gamma+k+1- (\gamma+1)\sqrt{k(k+1)}}{k+1-k\gamma^2}}\right] ^{\frac{1}{2}},\ &k \ne \frac{1}{\gamma^2-1},\ k>0\\
		\left[{\frac{\gamma-1}{2\gamma}}\right] ^{\frac{1}{2}},\ &k = \frac{1}{\gamma^2-1}
	\end{cases},\\
	M_2^*&=
	M_2^*(k)=\begin{cases}
		\left[ {\frac{k\gamma+k+1+ (\gamma+1)\sqrt{k(k+1)}}{k+1-k\gamma^2}}\right]^{\frac{1}{2}},\ & k < \frac{1}{\gamma^2-1}, \\
		\infty,\ & 0\leq k \geq \frac{1}{\gamma^2-1}
	\end{cases}.\\
	M_3^*&=
	M_3^*(k)=\left[ \frac{\sqrt{(k\gamma^2+1)(k+1)}-(k\gamma+1)}{k\gamma(\gamma-1)}\right] ^{\frac{1}{2}}, \ k>0.
	\end{aligned}
\end{equation*}
\begin{Proposition}\label{P2}
	The two functions \( \psi_1 \) and \( \psi_2 \) have the following properties:
	\begin{itemize}
		\item [(i)] 
		$\psi_1(M,k)\in C^\infty\left( (0,M_1^*] \cup [M_2^{*},\infty), (0,1]\right) $. \\
		$\psi_2(M,k)\in C^\infty \left( (M_3^*,M_1^*] \cup [M_2^*,\infty), [1,\infty)\right) $.
		 \item [(ii)] $\lim\limits_{M\to 0} \psi_1(M,k)=0$, $\lim\limits_{M\to M_3^*}\psi_2(M,k)=\infty$.\\
		 $\lim\limits_{M\to \infty} \psi_1(M,k)=0$, $\lim\limits_{M\to \infty}\psi_2(M,k)=\infty$.\\  $\psi_1(M_1^*,k)=\psi_2(M_1^*,k)=\psi_1(M_2^*,k)=\psi_2(M_2^*,k)=1$.
		 \item [(iii)] If $M\leq M_1^*(k)$, then $\frac{\partial \psi_1}{\partial M}(M,k)>0$ and $\frac{\partial \psi_2}{\partial M}(M,k)<0$;\\
		 If $M\geq M_2^{*}(k)$, then $\frac{\partial \psi_1}{\partial M}(M,k)<0$ and $\frac{\partial \psi_2}{\partial M}(M,k)>0$.
		 \item  [(iv)]
		 $\psi_2(M,k)=\sqrt{\frac{(\gamma-1)\psi_1^2(M,k)+2}{2\gamma \psi_1^2(M,k)+1-\gamma}}$,\quad 
		 $\psi_1(M,k)=\sqrt{\frac{(\gamma-1)\psi_2^2(M,k)+2}{2\gamma \psi_2^2(M,k)+1-\gamma}}$,\quad 
		 for any $k\geq 0$.
		 \item [(v)]
		 $M_2^*=\sqrt{\frac{(\gamma-1)(M_1^*)^2+2}{2\gamma (M_1^*)^2+1-\gamma}}$,\quad 
		 $M_1^*=\sqrt{\frac{(\gamma-1)(M_2^*)^2+2}{2\gamma (M_2^*)^2+1-\gamma}}$,\quad 
		 for any $k\geq 0$.
		 \item [(vi)]
		 For $i=1,2$, if $\psi_i(M_1,k)=\psi_i(M_2,k)$, then $M_1=\sqrt{\frac{(\gamma-1)M_2+2}{2\gamma M_2+1-\gamma}}$ and $M_2=\sqrt{\frac{(\gamma-1)M_1+2}{2\gamma M_1+1-\gamma}}$.
	\end{itemize}
\end{Proposition}

It is observed that, in general, a given left boundary state $\vU_-$ may correspond to two different right boundary states, and symmetrically, a given right boundary state may correspond to two different left boundary states.
The existence of multiple branches of boundary states potentially gives rise to non-uniqueness of the coupled Riemann solutions; however, such non-uniqueness is not an immediate consequence.

When constructing coupled Riemann solutions in the sense of Theorem~\ref{T1}, the admissible boundary states are required to belong to the sets $\cV_L(\vU_L)$ and $\cV_R(\vU_R)$.
It is important to note that one of the two branches corresponds to a subsonic state, while the other corresponds to a supersonic state, as shown in Proposition~\ref{P2} (i).
This difference directly influences the wave structure of the associated classical Riemann problem and, consequently, the validity of the conditions $\cV_L(\vU_L)$ and $\cV_R(\vU_R)$.
Therefore, the uniqueness of the coupled Riemann solution is determined by the admissible range of the boundary states.
In what follows, we show uniqueness of the Riemann solution in the case where the left-hand boundary state is supersonic.

\begin{Theorem}\label{E13}
	For any given Riemann data $\vU_L$ and $\vU_R$, there exists at most one weak solution satisfying the condition $M_-> 1$.
\end{Theorem}
\begin{proof}
	It suffices to show that the pair of boundary states $(\vU_-,\vU_+)$ is unique.  
	
	For the left classical Riemann problem $\vV\!\left( \tfrac{x}{t}, \vU_L,\vU_- \right)$, the Lax entropy condition implies successively that the waves associated with the $3$-, $2$-, and $1$-characteristic fields must all vanish. Hence, we must have $\vU_-=\vU_L$.  
	
	By Proposition~\ref{P2} (i), the right boundary state $\vU_+$ admits two possible candidates: a subsonic state, denoted by $\vU_{+}^1$, and a supersonic state, denoted by $\vU_{+}^2$.  
	We claim that $\vU_{+}^2 \notin \cV_R(\vU_R)$ whenever a solution with right boundary state $\vU_{+}^1$ exists.  
	
	Indeed, Proposition~\ref{P2} (iv) and uniqueness of the classical Riemann problem yield
	\[
	\vV\!\left(\tfrac{x}{t}; \vU_{+}^2,\vU_R\right) =
	\begin{cases}
		\vU_+^2, & x<0, \\[0.3em]
		\vV\!\left(\tfrac{x}{t}; \vU_{+}^1,\vU_R\right), & x>0 .
	\end{cases}
	\]
	Since $\vU_{+}^1 \in \cV_R(\vU_R)$, it follows that
	\[
	\vV\!\left(0+; \vU_{+}^2,\vU_R\right) 
	= \vV\!\left(0+; \vU_{+}^1,\vU_R\right)
	= \vU_+^1 \neq \vU_+^2 ,
	\]
	Therefore, $\vU_{+}^2 \notin \cV_R(\vU_R)$, the right boundary state $\vU_+$ is unique, and the theorem follows.
\end{proof}

However, the presence of multiple branches indeed leads to the non-uniqueness of the Riemann solutions with respect to the initial data.
More precisely, although Theorem~\ref{E13} ensures that the Riemann solutions with supersonic left boundary states are unique, alternative Riemann solutions with subsonic left boundary states can still be constructed.
Next, we establish the non-uniqueness of entropy solutions for a broad class of Riemann data.
This non-uniqueness originates from the inverse branch, which implies that, for a given right boundary state, two different left boundary states can exist.

\begin{Theorem}[\bf Non-uniqueness of entropy solutions]\label{E11}
	Let $k>0$ be fixed.
	Suppose $\vU$ is an entropy solution of the coupled Riemann problem with initial data $(\vU_L,\vU_R)$ satisfying $M_->1$ and $M_+>1$.
	Then there exists another entropy solution $\tvU$ with the same initial data such that $\tM_-<1$ and $\tM_+>1$.
\end{Theorem}

\begin{proof}
	We construct $\tvU$ explicitly.  
	
	By the Lax entropy condition we have $\vU_-=\vU_L$ with $M_L>1$.  
	Choose
	\[
	\overline{\vU}_-=L_1^+(\overline{\eps}, \vU_L), \quad \overline{\eps}>0,
	\]
	so that $\vU_L$ and $\overline{\vU}_-$ constitute a stationary shock wave. Consequently,
	\[
	\psi_2(\overline{\vU}_-,k)=\vU_+ .
	\]
	
	Next, perturb slightly along the 1-wave curve:
	\[
	\tvU_- = L_1^+(\overline{\eps}+\eps, \vU_L), \quad \eps>0,
	\]
	and set $\tvU_+=\psi_2(\tvU_-,k)$.  
	We claim that $\tvU_- \in \cV_L(\vU_L)$ and $\tvU_+ \in \cV_R(\vU_R)$.  
	
	Since $\eps$ is sufficiently small, $\vU_L$ and $\tvU_-$ are connected by a shock wave with speed
	\[
	\begin{aligned}
		\sigma
		&= u_L - c_L \sqrt{\Bigl(\tfrac{\gamma+1}{2\gamma}\Bigr)\Bigl(\tfrac{\widetilde{p}^-}{p_L}\Bigr) + \tfrac{\gamma-1}{2\gamma}} \\
		&< u_L - c_L \sqrt{\Bigl(\tfrac{\gamma+1}{2\gamma}\Bigr)\Bigl(\tfrac{\overline{p}^-}{p_L}\Bigr) + \tfrac{\gamma-1}{2\gamma}} = 0 ,
	\end{aligned}
	\]
	which implies $\tvU_- \in \cV_L(\vU_L)$.  
	
	Moreover,
	\[
	\|\tvU_+ - \vU_+\|
	= \|\psi_2(\tvU_-,k)-\psi_2(\overline{\vU}_-,k)\|
	\aleq \|L_1^+(\overline{\eps}+\eps,\vU_L)-L_1^+(\overline{\eps},\vU_L)\|
	\aleq |\eps| .
	\]
	By Proposition~\ref{P3}, the classical Riemann solution $\vV(\tfrac{x}{t}; \tvU_+,\vU_R)$ exists, with intermediate states close to those of $\vV(\tfrac{x}{t}; \vU_+,\vU_R)$.  
	Since $M_+>1$ and $\vU_+\in \cV_R(\vU_R)$, it follows that $\tvU_+\in \cV_R(\vU_R)$.  
	
	Hence, the pair $(\tvU_-,\tvU_+)$ yields another entropy solution $\tvU$, completing the proof.
\end{proof}

\section{Admissible Weak Solutions}\label{sec4}
In view of Theorem~\ref{E11}, additional admissibility criterion is needed to select physically-relevant solutions. In the present work, we employ the evolutionarity criterion of Landau and Lifshitz~\cite{landau1987fluid}, originally formulated for general discontinuities in fluid dynamics. We further note its successful application to duct flows in the work of Warnecke and Andrianov~\cite{warnecke2004solution}.

\begin{Definition}[\bf Evolutionarity criterion]
	Consider a discontinuity $\Sigma$ in a physical flow, which is governed by a $d\times d$ hyperbolic system. Denote the number of charactristics incoming to $\Sigma$ by $n$ and coinciding with $\Sigma$ by $c$, further, denote the number of unknown variables on both sides of $\Sigma$ together with the speed of $\Sigma$ by $N=2d+1$, and the number of relations across $\Sigma$ by $m$, then $\Sigma$ is called evolutionary if 
	\begin{equation*}
		N=n+c+m.
	\end{equation*}
\end{Definition}

In the present framework, the coupling interface is regarded as a stationary discontinuity across which three relations, referred to as the coupling conditions, are imposed. According to the evolutionarity criterion, exactly three characteristics must either enter or coincide with the interface. For the Euler system, the characteristic fields are associated with the eigenvalues
\begin{equation*}
	\lambda_1(\vU)=u-c,\quad \lambda_2(\vU)=u, \quad \lambda_3(\vU)=u+c.
\end{equation*}

Under the assumption $u_->0$ and $u_+>0$, the characteristics $\lambda_2$ and $\lambda_3$ are always incoming from the left state $\vU_-$ and outgoing from the right state $\vU_+$. Consequently, for a given pair of boundary states, exactly one characteristic is incoming to or coinciding with the interface. By invoking the evolutionarity criterion, the admissible entropy solutions of the coupled Riemann problem are thus required to satisfy the following admissibility condition.

\begin{Definition}[\bf Monotonicity criterion]\label{MC1}
	For any fixed parameter $k>0$, the left and right boundary states $\vU_-$ and $\vU_+$ are required to satisfy
	\begin{equation*}
		\lambda_i(U_-)\cdot \lambda_i(U_+)\geq 0, \quad \forall i=1,2,3
	\end{equation*}
	where $\lambda_i(\cdot)$ denote the characteristic speeds of the Euler system.
\end{Definition}

Observe that the case $\lambda_1(\vU_-)= \lambda_1(\vU_+)=0$ cannot occur when assuming $k\ne 0$. 
Consequently, the monotonicity criterion is equivalent to the evolutionarity criterion. 
In this framework, admissible solutions are understood as entropy weak solutions that additionally satisfy the monotonicity criterion.

The monotonicity criterion naturally arises in the classical theory of one-dimensional flow with heat addition; see \cite[Section~3.8]{anderson1990modern}. 
The Rayleigh curve shows that heat addition always drives the flow toward the sonic state, which corresponds to the state of maximum entropy. 
However, the flow cannot cross the sonic point solely through heat addition. 
For a fixed upstream state, there exists a maximum admissible heat input that brings the downstream state exactly to sonic conditions. 
When this limit is attained, the flow is said to be \emph{choked}.

We select the branch of the Mach number that satisfies the monotonicity criterion:
\begin{equation*}
	M_+=\psi(M_-,k):=
	\begin{cases}
		\psi_1(M_-,k), & M_-\leq M_1^*, \\
		\psi_2(M_-,k), & M_-\geq M_2^*.
	\end{cases}
\end{equation*}
By Proposition~\ref{P2} (iii), we have $\frac{\partial \psi (M,k)}{\partial M}>0$, and it is easy to check that $\psi(M_-;0)=M_-$. 
Once the admissible Mach number $M_+$ is determined, the corresponding admissible density $\rho_+$, velocity $u_+$, and pressure $p_+$ follow. 
For convenience, we introduce the functions
\begin{align*}
	h_1(M,k)&:=\frac{\gamma M^2+1}{\gamma (\psi(M,k))^2+1}, \\
	h_2(M,k)&:=\frac{(\psi(M,k))^2(\gamma M^2+1)}{M^2(\gamma (\psi(M,k))^2+1)}.
\end{align*}
Then the admissible boundary states satisfy
\begin{equation}\label{E12}
	\frac{p_+}{p_-}=h_1(M_-,k)<1,\qquad 
	\frac{u_+}{u_-}=h_2(M_-,k)>1.
\end{equation}

\begin{Proposition}\label{P4}
	Suppose that $0<M<M_1^*(k)$ and $0<k(\gamma^2-1)<1$. Then
	\begin{equation}\label{E10}
		\frac{\partial  h_1(M,k)}{\partial M}<0, \qquad 
		\frac{\partial  h_2(M,k)}{\partial M}>0.
	\end{equation}
\end{Proposition}

The proof is deferred to Appendix~\ref{App1}.

The monotonicity criterion imposes constraints on the eigenvalues of the boundary states, thereby limiting the structure of the Riemann solutions to several possible configurations.

\begin{Theorem}\label{E9}
	For any given Riemann data $\vU_L$ and $\vU_R$, there exist exactly three structures of the admissible Riemann solutions for $k\ne 0$. The corresponding boundary states for each case are given as follows:
	\begin{itemize}
		\item \textit{Structure 1:} 
		\begin{align*}
			&\vU_-=L_1^+(\vU_L,\eps_1), \quad \lambda_1(\vU_-)<0,\\ 
			&\vU_+=L_2^-(\eps_3;L_3^-(\eps_4;\vU_R)), \ \lambda_1(\vU_+)<0.
		\end{align*}
		\item \textit{Structure 2:} 
		\begin{align*}
			&\vU_-=L_1^+(\vU_L,\eps_1),\quad \lambda_1(\vU_-)<0, \\ &\vU_+=L_1^-(\eps_2;L_2^-(\eps_3;L_3^-(\eps_4;\vU_R))),\quad
			\eps_2\leq 0,\quad \lambda_1(\vU_+)=0.
		\end{align*}
		\item \textit{Structure 3:} 
		\begin{align*}
			&\vU_-=\vU_L,\quad \lambda_1(\vU_-)>0, \\
			&\vU_+=L_1^{-}(\eps_2;L_2^-(\eps_3;L_3^-(\eps_4;\vU_R))),\quad \lambda_1(\vU_+)\geq 0.
		\end{align*}
	\end{itemize}
	Here $\eps_i\,(i=1,2,3,4)$ are parameters ensuring that the states belong to the admissible domain.
\end{Theorem}

A similar necessity argument can be found in \cite{yu2024mathematical}, and the existence of each admissible structure follows from Theorem~\ref{E7}.

The case with constant initial data, i.e., $\vU_L = \vU_R$, is of particular interest. It has been applied to study wave generation due to instant condensation, see \cite{cheng2010condensation}. In this special situation, the monotonicity criterion determines the wave types explicitly. The following conclusion has been proven in \cite{yu2022riemann}.

\begin{Corollary}
	For constant initial data $\vU_L = \vU_R$, it always holds that $\varepsilon_1 \geq 0$, $\varepsilon_2 \leq 0$, and $\varepsilon_4 \geq 0$. Consequently, the 1-wave to the left of the boundary and the 3-wave to the right of the boundary must be shock waves, whereas the 1-wave to the right of the boundary must be a rarefaction wave.
\end{Corollary}

\section{Local Existence}\label{sec5}

\begin{Theorem}[\bf Local existence of admissible Riemann solutions]\label{E7}
	Let $\vU_L$ and $\vU_R$ be initial Riemann data satisfying condition~\eqref{E6}, and assume that the corresponding classical Riemann problem admits intermediate states with strictly positive velocity. Then, for all sufficiently small $k>0$, there exists an admissible solution to the coupled Riemann problem.
\end{Theorem}

\begin{proof}
	
	Let $\vV = \vV\!\left( \tfrac{x}{t}; \vU_L, \vU_R \right)$ be the classical Riemann solution with left and right intermediate states $\vV_{*L}$ and $\vV_{*R}$. 
	Along the $t$-axis, the state of $\vV$ falls into one of the following cases:
	\begin{itemize}
		\item[(i)] $\vV(0;\vU_L,\vU_R)=\vV_{*L}$ with $\lambda_1(\vV_{*L})<0$;
		\item[(ii)] the 1-wave is a rarefaction and the $t$-axis lies inside it, i.e.\ $\lambda_1(\vU_{L})\leq 0<\lambda_1(\vV_{*L})$;
		\item[(iii)] $\vV(0;\vU_L,\vU_R)=\vU_L$ with $\lambda_1(\vU_L)>0$;
		\item[(iv)] a stationary 1-shock;
		\item[(v)] $\vV(0;\vU_L,\vU_R)=\vV_{*L}$ with $\lambda_1(\vV_{L})\leq 0,\ \lambda_1(\vV_{*L})=0$.
	\end{itemize}
	
	The proof proceeds by first identifying the boundary states $(\vU_-,\vU_+)$ and then verifying that 
	\[
	\vU_- \in \mathcal{V}_L(\vU_L), \qquad \vU_+ \in \mathcal{V}_R(\vU_R).
	\]
	In Cases (ii)–(iii), $\vU_-$ and $\vU_+$ can be determined directly.  
	In Cases (i),(iv) and (v), we apply the implicit function theorem (for Lipschitz continuous functions).  
	We now treat each case separately to establish local existence of admissible Riemann solutions.
	
	\textbf{Case (i).}
	We establish the existence of an admissible Riemann solution of \textit{Structure 1} for small $k>0$.
	Since the coupling condition is expressed in terms of the Mach number, it is convenient to represent the 1-wave using it.
	Consider the 1-wave connecting the left state $\vU_L$ to the intermediate state $\vU_-$. We have
	\begin{equation}
		u_-(p_-;\vU_L)=
		\begin{cases}
			u_L-\dfrac{p_--p_L}{\rho_L}\left( \dfrac{2}{(\gamma+1)p_-+(\gamma-1)p_L} \right)^{\tfrac{1}{2}}, & p_-\geq p_L, \\[1ex]
			u_L-\dfrac{2c_L}{\gamma-1}\left[ \left( \tfrac{p_-}{p_L}\right) ^{\tfrac{\gamma-1}{2\gamma}}-1 \right], & p_-< p_L ,
		\end{cases}
	\end{equation}
	and
	\begin{equation}
		c_-(p_-;\vU_L)=
		\begin{cases}
			c_L\sqrt{\tfrac{p_-}{p_L}}\sqrt{\dfrac{(\gamma-1)p_-+(\gamma+1)p_L}{(\gamma-1)p_L+(\gamma+1)p_-}}, & p_-\geq p_L, \\[1ex]
			c_L\left( \tfrac{p_-}{p_L} \right) ^{\tfrac{\gamma-1}{2\gamma}}, & p_-< p_L .
		\end{cases}
	\end{equation}
	The admissible domain is specified to ensure $u_->0$:
	\begin{equation}
		p_-\in 
		\begin{cases}
			\left(0,\; p_L\Big(1+\tfrac{(\gamma-1)u_L}{2c_L}\Big)^{\tfrac{2\gamma}{\gamma-1}} \right), & u_L\leq 0, \\[1ex]
			\left(0,\; p_L+\tfrac{1}{4}\big((\gamma+1)p_L^2u_L^2+p_Lu_L\sqrt{(\gamma+1)^2p_L^2u_L^2+16\gamma p_L}\big)\right), & u_L>0 .
		\end{cases}
	\end{equation}
	It is straightforward to check that $u_-(p_-;\vU_L)$ and $c_-(p_-;\vU_L)$ are differentiable in $p_-$, with $(u_-)'<0$ and $(c_-)'>0$, cf.~\cite{godlewski2013numerical}.  
	The Mach number is
	\[
	M_-(p_-;\vU_L)=\frac{u_-(p_-;\vU_L)}{c_-(p_-;\vU_L)}, \qquad (M_-)'<0,
	\]
	whose range is $(0,\infty)$.  
	Hence, there exists an inverse $g_1:(0,\infty)\to \R_+$ with
	\[
	p_-=g_1(M_-;\vU_L),\qquad g_1' = \frac{1}{(M_-)'}<0.
	\]
	The velocity can then be written as $u_-=g_2(M_-;\vU_L)$, where
	\[
	g_2'=\frac{du_-}{dp_-}\frac{dp_-}{dM_-}>0.
	\]
	
	For the boundary states $(\vU_-,\vU_+)$, we extend their relations to $k< 0$. For brevity, we retain the same notation:
	\begin{align*}
		&\psi_1(M,k)\in C^1\big((0,M_1^*(k)]\times \R_+ \cup (0,1]\times \R_-;\R_+\big),
		&& \psi_1(M,k)=\sqrt{\tfrac{1-I}{1+\gamma I}},\\
		&h_1(M,k)\in C^1\big((0,M_1^*(k)]\times \R_+ \cup (0,1]\times \R_-;\R_+\big),
		&& h_1(M,k)=\frac{\gamma M^2+1}{\gamma (\psi_1(M,k))^2+1}, \\
		&h_2(M,k)\in C^1\big((0,M_1^*(k)]\times \R_+ \cup (0,1]\times \R_-;\R_+\big),
		&& h_2(M,k)=\frac{(\psi_1(M,k))^2(\gamma M^2+1)}{M^2(\gamma (\psi_1(M,k))^2+1)} .
	\end{align*}
	They are well-defined, and satisfy
	\[
	\psi_1(M;0)=M,\quad h_1(M;0)=h_2(M;0)=1,\quad 
	\frac{d\psi_1(M;0)}{dM}=1,\ \frac{dh_1(M;0)}{dM}=\frac{dh_2(M;0)}{dM}=0.
	\]
	The relations for the states $\vU_-$ and $\vU_+$ are
	\[
	M_+=\psi_1(M_-,k),\qquad p_+=p_-\, h_1(M_-,k),\qquad u_+=u_-\, h_2(M_-,k).
	\]
	
	We now combine all the waves in the Riemann solution of \textit{Structure~1}.  
	Define
	\begin{equation*}
		\begin{aligned}
			T_1(M,k) &\in C^1\big(\mathcal{O}_1(M_{*L})\times \R;\R_+\big),\\
			T_1(M,k) &:= g_2(M;\vU_L)h_2(M,k)-f_R\big(g_1(M;\vU_L)h_1(M,k);\vU_R\big),
		\end{aligned}
	\end{equation*}
	where $\mathcal{O}_1(M_{*L})$ is a neighbourhood of $M_{*L}$, and $M_{*L}$ is the Mach number of $\vV_{*L}$.
	By Proposition~\ref{P2} (i) and (ii), $T_1$ is well-defined.
	From Proposition~\ref{P4}, $M_{*L}$ is a root of $T_1$ when $k=0$, and
	\[
	\frac{\partial T_1(M,0)}{\partial M}=g_2'(M;\vU_L)-f_R'(g_1(M;\vU_L);\vU_R)g_1'(M;\vU_L)>0.
	\]
	Hence, by the implicit function theorem, there exist neighbourhoods $\mathcal{O}(M_{*L})\subset \mathcal{O}_1(M_{*L})$, $\mathcal{O}(0)\subset \R$, and a $C^1$ map $M:\mathcal{O}(0)\to \mathcal{O}(M_{*L})$ such that
	\[
	T_1(M(k),k)=0, \qquad \forall k\in \mathcal{O}(0).
	\]
	For any $k\in \mathcal{O}(0)\cap \R_+$, an admissible Riemann solution of \textit{Structure 1} is constructed as follows.  
	We choose $\vU_-\in L_1^+(\eps,\vU_L)$ with $M_-=M(k)\in \mathcal{O}(M_{*L})$ and the corresponding $\vU_+$. Then
	\[
		T_1(M_-,k)=u_+-f_R(p_+;\vU_R)=0,
	\]
	which implies that there exist $\eps_3,\eps_4$ such that
	\[
	\vU_+=L_2^+(\eps_3;\,L_3^+(\eps_4;\vU_R)).
	\]
	Since $u_+>0$, it follows that $\vU_+\in \mathcal{V}_R(\vU_R)$.  
	Moreover, the 1-wave connecting $\vU_L$ to $\vV_{*L}$ has negative speed (by the assumption of Case~(i)). By shrinking $\mathcal{O}(M_{*L})$ if necessary, the distance $\|\vU_- - \vV_{*L}\|$ can be made arbitrarily small, ensuring $\vU_-\in \mathcal{V}_L(\vU_L)$.
	
	\textbf{Case (ii).}  
	For any suffciently small $k>0$, an admissible Riemann solution of \textit{Structure 2} is constructed as follows.  
	We choose $\vU_-=L_1^+(\eps_1;\vU_L)$ such that $M_-=M_1^*<1$.  
	If $\lambda_1(\vU_L)=0$, then the wave connecting $\vU_L$ and $\vU_-$ is a non-degenerate shock with speed
	\[
	\sigma < \lambda_1(\vU_L)=0.
	\]
	If $\lambda_1(\vU_L)<0$, then by Proposition~\ref{P2} (ii) it is a rarefaction wave for sufficiently small $k>0$.  
	In both cases we obtain $\vU_-\in \mathcal{V}_L(\vU_L)$.
	
	The corresponding right boundary state $\vU_+$ satisfies $M_+=1$. Moreover,
	\begin{align*}
		\big\| \vU_+ - \vV(0;\vU_L,\vU_R) \big\|
		&\leq \|\vU_+-\vU_-\| + \|\vU_- - \vV(0;\vU_L,\vU_R)\| \\
		&\aleq \|\vU_+-\vU_-\| + |M_1^*(k)-1| \\
		&\aleq k.
	\end{align*}
	Hence, by Proposition~\ref{P3}, the classical RP solution $\vV(0;\vU_+,\vU_R)$ exists and contains a non-degenerate rarefaction under the assumption $\lambda_1(\vV_{*L})>0$.  
	Thus, $\vU_+\in \mathcal{V}_R(\vU_R)$.
	
	\textbf{Case (iii).}  
	For any suffciently small $k>0$, an admissible Riemann solution of \textit{Structure 3} is constructed as follows.  
	Here the initial data satisfies $M_L>1$. By Proposition~\ref{P2}(ii), we have $M_L > M_2^*$ for sufficiently small $k>0$.  
	Thus we can take $\vU_-=\vU_L$ and determine the corresponding $\vU_+$.  
	Since 
	\[
	\|\vU_L-\vU_+\| = \|\vU_--\vU_+\| \aleq k, 
	\qquad 
	u_L+\frac{2c_L}{\gamma-1}>u_R-\frac{2c_R}{\gamma-1},
	\]
	it follows that
	\[
	u_+ + \frac{2c_+}{\gamma-1} > u_R - \frac{2c_R}{\gamma-1}
	\]
	for sufficiently small $k$.  
	By Proposition~\ref{P3}, the classical RP solution $\vV(x/t;\vU_+,\vU_R)$ exists and is unique.  
	Moreover, the intermediate states of $\vV(x/t;\vU_L,\vU_R)$ and $\vV(x/t;\vU_+,\vU_R)$ differ only by $O(k)$.  
	Since $M_{*L}>1$, we conclude that $\vU_+\in \mathcal{V}_R(\vU_R)$.  
	
	\textbf{Case (iv).} 
	This case may correspond to two possible structures (\textit{Structure 1} or \textit{Structure 3}). We therefore introduce the piecewise function
	\begin{align*}
		&T_4(M,k) := \\
		&\begin{cases}
			g_2 (\psi^{-1}(M,k);\vU_L)\, h_2 (\psi^{-1}(M,k),k)
			-f_R\left( g_1 (\psi^{-1}(M,k);\vU_L)\, h_1 (\psi^{-1}(M,k),k) \right), & M\leq \psi_1(M_L,k), \\[1ex]
			g_2(M;\tvU_+)-f_R(g_1(M;\tvU_+);\vU_R), & M>\psi_1(M_L,k),
		\end{cases}
	\end{align*}
	where the state $\tvU_+$ is given by
	\[
	\widetilde{M}_+=\psi_1(M_L,k),\quad 
	\widetilde{p}_+=p_-\cdot h_1(M_L,k),\quad 
	\widetilde{u}_+=u_-\cdot h_2(M_L,k).
	\]
	$\psi^{-1}(M,k)$ is the inverse function of the function $\psi(M,k)$ with respect to $M$ for a given $k$.
	It is easy to check that $T_4(M,k)$ is Lipschitz continuous with respect to $M$:
	\begin{align*}
				\lim\limits_{M\to \psi_1(M_L,k)-}T_4(M,k)&=g_2\circ \psi(M;0)\cdot h_2\circ \psi(M;0)-f_R\left( g_1\circ \psi(M;0)\cdot h_1\circ \psi(M;0) \right) ,\\
				&= \lim\limits_{M\to \psi_1(M_L,k)+}T_4(M,k),
	\end{align*}
	and the one-sided derivatives at $M=\psi_1(M_L,k)$ satisfy
	\begin{align*}
				&\lim\limits_{M\to \psi_1(M_L,0)-}\frac{\partial T_4(M,0)}{\partial M}=g_2' (\psi^{-1}(M,k);\vU_L)-f_R'\left( g_1 (\psi^{-1}(M;0);\vU_L) \right) g_1' (\psi^{-1}(M,k);\vU_L)>0,\\
				&\lim\limits_{M\to \psi_1(M_L,0)+}\frac{\partial T_4(M,0)}{\partial M}=g_2'(M;\tvU_+)-f_R'(g_1(M;\tvU_+);\vU_R)g_1'(M;\tvU_+)>0
	\end{align*}
	Moreover, by assumption,
	\[
	T_4(M_{*L},0)=u_{*L}-f_R(p_{*L};\vU_R)=0.
	\]
	Hence, by the implicit function theorem for Lipschitz functions \cite[Chapter~7]{clarke1990optimization}, there exists a $C^1$ function $M(k)$ with $M(0)=M_{*L}$ such that $T_4(M(k),k)=0$ for small $k>0$.
	Next, we construct an admissible Riemann solution as follows.
	
	\medskip
	\noindent\emph{Structure 1.}  
	If $M(k)\leq \psi_1(M_L,k)$, choose $\vU_-\in L_1^+(\eps_1;\vU_L)$ with $M_-=\psi^{-1}(M(k))$ and corresponding $\vU_+$. The $1$-wave connecting $\vU_L$ and $\vU_-$ has speed
	\begin{align*}
				\sigma=&u_L-c_L\sqrt{\left( \frac{\gamma+1}{2\gamma}\right)  \left( \frac{{ p}^-}{p_L} \right) +\left( \frac{\gamma-1}{2 \gamma} \right) }\\
				=&u_L-c_L\sqrt{\left( \frac{\gamma+1}{2\gamma}\right)  \left( \frac{g_1(\psi^{-1}(M(k),k);\vU_L)}{p_L} \right) +\left( \frac{\gamma-1}{2 \gamma} \right) }\\
				\leq &u_L-c_L\sqrt{\left( \frac{\gamma+1}{2\gamma}\right)  \left( \frac{g_1(\psi^{-1}(\psi_1(M_L,k));\vU_L)}{p_L} \right) +\left( \frac{\gamma-1}{2 \gamma} \right) }\\
				= &u_L-c_L\sqrt{\left( \frac{\gamma+1}{2\gamma}\right)  \left( \frac{g_1(\psi_1(M_L,k);\vU_L)}{p_L} \right) +\left( \frac{\gamma-1}{2 \gamma} \right) }\qquad (\text{ by Proposition~\ref{P2} (vi)})\\
				=&0.
	\end{align*}
	We claim $\sigma\neq 0$. Otherwise $\vU_-=\vV_{*L}$ and 
	\[
	u_--f_R(p_-;\vU_R)=0.
	\]
	However, since $T_4(M(k),k)=u_+-f_R(p_+;\vU_R)=0$ and $k>0$ implies $(p_-,u_-)\neq(p_+,u_+)$, monotonicity of $f_R$ yields a contradiction. Thus $\sigma\neq 0$, proving $\vU_-\in \cV_L(\vU_L)$ and $\vU_+\in \cV_R(\vU_R)$.
	
	\medskip
	\noindent\emph{Structure 3.}  
	If $M(k)>\psi_1(M_L,k)$, set $\vU_-=\vU_L\in\cV_L(\vU_L)$. The corresponding $\vU_+$ satisfies
	\[
	T_4(M(k),k)=g_2(M(k);\vU_+)-f_R(g_1(M(k);\vU_+);\vU_R)=0,
	\]
	so that $\vV(\tfrac{x}{t};\vU_+,\vU_R)$ exists with left intermediate Mach number $M(k)$. Since $M(k)<1<M_+$, the $1$-wave in this RP is a shock wave with speed
	\begin{align*}
				\sigma=&u_+-c_+\sqrt{\left( \frac{\gamma+1}{2\gamma}\right)  \left( \frac{g_1(M(k));\vU_+}{p_+} \right) +\left( \frac{\gamma-1}{2 \gamma} \right) }\\
				> &u_+-c_+\sqrt{\left( \frac{\gamma+1}{2\gamma}\right)  \left( \frac{g_1(\psi_1(M_L,k);\vU_+)}{p_+} \right) +\left( \frac{\gamma-1}{2 \gamma} \right) }\\
				= &0\qquad \qquad \qquad \qquad (\text{ by Proposition~\ref{P2} (vi)})
	\end{align*}
	which shows $\vU_+\in\cV_R(\vU_R)$.
	
	\textbf{Case (v).}  
	This case may correspond to either \textit{Structure 1} or \textit{Structure 2}. Define
	\begin{align*}
		&T_5(M,k):=\\
		&\begin{cases}
			g_2 (\psi^{-1}(M,k);\vU_L)\, h_2 (\psi^{-1}(M,k),k)
			-f_R\!\left( g_1 (\psi^{-1}(M,k);\vU_R)\, h_1 (\psi^{-1}(M,k),k) \right), & M<1, \\[1ex]
			g_2(M;\tvU_+)-f_R(g_1(M;\tvU_+);\vU_R), & M\geq 1,
		\end{cases}
	\end{align*}
	where $\tvU_+$ is obtained as follows: choose
	\[
	\tvU_-:=L_1^+(\eps_1;\vU_L)\quad \text{with}\quad \tM_-=M_1^*(k),
	\]
	and set
	\[
	\tM_+=\psi(\tM_-,k),\qquad \tp_+=\tp_-\, h_1(\tM_-,k),\qquad \tu_+=\tu_-\, h_2(\tM_-,k).
	\]
	
	It follows that
	\[
	\lim_{M\to 1} T_5(M,k)=\tu_+-f_R(\tp_+;\vU_R)=T_5(1,k),
	\]
	so $T_5$ is continuous at $M=1$. Moreover, the one-sided derivatives at $(M,k)=(1,0)$ satisfy
	\[
	\lim_{M\to 1-}\frac{\partial T_5(M;0)}{\partial M}>0, \qquad 
	\lim_{M\to 1+}\frac{\partial T_5(M;0)}{\partial M}>0.
	\]
	Since $T_5(1,0)=0$, the implicit function theorem for Lipschitz functions ensures the existence of a $C^1$ map $M(k)$, defined for small $k>0$, such that $T_5(M(k),k)=0$.
	
	\medskip
	\noindent\emph{Structure 1.}  
	If $M(k)<1$, take $\vU_-=L_1^+(\eps;\vU_L)$ with $M_-=\psi^{-1}(M(k))$.  
	Since $M_L\leq 1$ and $M_-<1$, the $1$-wave connecting $\vU_L$ and $\vU_-$ has negative speed, so $\vU_-\in\cV_L(\vU_L)$.  
	The corresponding $\vU_+$ satisfies
	\[
	u_+-f_R(p_+;\vU_R)=0,
	\]
	hence $\vU_+\in\cV_R(\vU_R)$.
	
	\medskip
	\noindent\emph{Structure 2.}  
	If $M(k)\geq 1$, choose $(\vU_-,\vU_+)=(\tvU_-,\tvU_+)$ as defined above.  
	The $1$-wave connecting $\vU_L$ and $\vU_-$ has negative speed, so $\vU_-\in\cV_L(\vU_L)$.  
	Moreover,
	\[
	g_2(M(k);\vU_+)-f_R(g_1(M(k);\vU_+);\vU_R)=0,
	\]
	which guarantees that the classical RP $\vV(x/t;\vU_+,\vU_R)$ exists with left intermediate Mach number $M(k)$.  
	Since $M_+=1$ and $M(k)\geq 1$, this RP contains a $1$-rarefaction wave, thus $\vU_+\in\cV_R(\vU_R)$.
\end{proof}

\begin{Corollary}[\bf Uniqueness of each structure]\label{Co1}
	Assume that $k<1/(\gamma^2-1)$. Then, for any given Riemann data $(\vU_L,\vU_R)$, there exists at most one admissible solution for each structure.
\end{Corollary}

\begin{proof}
	The uniqueness of \textit{Structure~2} and \textit{Structure~3} follows directly from the monotonicity of $\psi$ and the uniqueness of the classical Riemann solution. 
	
	For \textit{Structure~1}, recall from the proof of case~(i) in Theorem~\ref{E7} that the left boundary state satisfies
	\[
	T_1(M_-,k)=g_2(M_-;\vU_L)\,h_2(M_-,k)-f_R\!\left(g_1(M_-;\vU_L)\,h_1(M_-,k);\vU_R\right)=0.
	\]
	By Proposition~\ref{P4}, its derivative is strictly positive:
	\begin{align*}
		\frac{\partial T_1}{\partial M}(M_-,k)
		&= g_2'(M_-;\vU_L)h_2(M_-,k)+g_2(M_-;\vU_L)\,\frac{\partial h_2}{\partial M}(M_-,k) \\
		&\quad - f_R'\left(g_1(M_-;\vU_L)h_1(M_-,k);\vU_R\right)
		\big(g_1'(M_-;\vU_L)h_1(M_-,k)+g_1(M_-;\vU_L)\frac{\partial h_1}{\partial M}(M_-,k)\big) \\
		&>0,
	\end{align*}
	which implies that the Mach number $M_-$ is uniquely determined. Together with the monotonicity of $\psi$ and the uniqueness of the classical Riemann solution, this yields the uniqueness of the admissible Riemann solution.
\end{proof}

\begin{Remark}\label{R1}
	By the proof of Theorem~\ref{E7}, the admissible Riemann solutions depend continuously on the unique solution of the classical Riemann problem when $k>0$ is sufficiently small. Therefore, they are \emph{physically admissible} in the sense of~\cite{herty2019coupling}.
\end{Remark}

\section{Non-existence}\label{sec6}

\begin{Theorem}[\bf Non-existence of admissible Riemann solutions]\label{E8}
	For any fixed $k>0$, there exist infinitely many Riemann data $\vU_L$ and $\vU_R$ such that no admissible Riemann solution exists.
\end{Theorem}
\begin{proof}
	We construct a family of Riemann data for which no admissible solution exists.
	Choose a state $\vU_L\in \cD$ with $M_L>M_2^*(k)$, and define
	\begin{equation*}
		\vU_R=L_3^+(\eps_3;L_2^+(\eps_2;\oline \vU)),
	\end{equation*}
	where $\oline \vU$ is given by the coupling condition $\Psi(\vU_L,\oline \vU)=0$, which violates the monotonicity criterion, i.e. $\oline M=\psi_1(M_L,k)$. 
	The free parameters $\eps_2,\eps_3$ are chosen so that the states remain in $\cD$.
	
	\medskip
	\noindent{Step 1. Non-existence of Structure~1.}  
	Let $\vU_-=L_1^+(\eps;\vU_L)$ with
	\begin{equation*}
		p_-=p_L\cdot\frac{2\gamma M_L^2-\gamma +1}{\gamma+1}.
	\end{equation*}
	Since $M_L>M_2^*(k)>1$, one checks that
	\begin{equation*}
		p_->p_L, 
		\qquad 
		0<M_-=\frac{(\gamma-1)M_L^2+2}{2\gamma M_L^2-\gamma+1}<1<M_L,
	\end{equation*}
	so that the 1-wave connecting $\vU_L$ and $\vU_-$ is a shock wave with speed
	\begin{equation*}
		\sigma=u_L-c_L\sqrt{\frac{\gamma+1}{2\gamma}\cdot\frac{p_-}{p_L}+\frac{\gamma-1}{2\gamma}}=0.
	\end{equation*}
	For the corresponding right boundary state $\vU_+$, since $M_+\le 1$ and $\oline M\le 1$, it follows from Proposition~\ref{P2} (iv) that $\vU_+=\oline \vU$.  
	Moreover, $M_-$ is the unique root of $T_1(M,k)=0$, see the proof of Corollary~\ref{Co1}.  
	However, $\sigma=0$ implies $\vU_-\notin \cV_L(\vU_L)$. Hence Structure~1 is impossible.
	
	\medskip
	\noindent{Step 2. Non-existence of Structure~2.}  
	Assume otherwise. Then the left boundary state must be
	\begin{equation*}
		\vU_-=L_1^+(\eps;\vU_L), \qquad M_-=M_1^*.
	\end{equation*}
	By Proposition~\ref{P2} (v),
	\begin{equation*}
		\frac{(\gamma-1)M_L^2+2}{2\gamma M_L^2+1-\gamma}
		<\frac{(\gamma-1)(M_2^*)^2+2}{2\gamma (M_2^*)^2+1-\gamma}=M_1^*.
	\end{equation*}
	Thus the left 1-wave is a shock with speed
	\begin{align*}
		\sigma
		&=u_L-c_L\sqrt{\tfrac{\gamma+1}{2\gamma}\cdot\tfrac{g_1(M_1^*)}{p_L}+\tfrac{\gamma-1}{2\gamma}} \\
		&>u_L-c_L\sqrt{\tfrac{\gamma+1}{2\gamma}\cdot \tfrac{1}{p_L}g_1\!\left(\tfrac{(\gamma-1)M_L^2+2}{2\gamma M_L^2+1-\gamma}\right)+\tfrac{\gamma-1}{2\gamma}} \\
		&=u_L-c_L\sqrt{\tfrac{\gamma+1}{2\gamma}\cdot\tfrac{1}{p_L}\,p_L\cdot\tfrac{2\gamma M_L^2-\gamma+1}{\gamma+1}+\tfrac{\gamma-1}{2\gamma}} \\
		&=0.
	\end{align*}
	Therefore $\vU_-\notin \cV_L(\vU_L)$, and Structure~2 is impossible.
	
	\medskip
	\noindent{Step 3. Non-existence of Structure~3.}  
	Assume otherwise. Then $\vU_-=\vU_L$. Since $\oline M<1$ and $M_+>1$, Proposition~\ref{P2} (iv) implies that the pair $(\vU_+,\oline \vU)$ forms a stationary shock. With 
	\[
	\vU_R=L_3^+(\eps_3;L_2^+(\eps_2;\oline \vU)),
	\]
	the associated classical Riemann solution reads
	\[
	\vV\!\left(\tfrac{x}{t};\vU_+,\vU_R\right)=
	\begin{cases}
		\vU_+ , & x<0,\\
		\vV\!\left(\tfrac{x}{t};\oline \vU,\vU_R\right), & x>0.
	\end{cases}
	\]
	Hence $\vU_+\notin \cV_R(\vU_R)$. Thus Structure~3 is also impossible.
	
	\medskip
	Combining the above steps, none of the three admissible structures can occur. 
	Therefore, for the constructed data, no admissible Riemann solution exists.
\end{proof}

\begin{Remark}\label{R2}
	For the initial data in Theorem~\ref{E8}, there always exists an entropy solution of the form
	\begin{equation*}
		\vU(t,x)=
		\begin{cases}
			\vU_L, & x<0,\\[0.3em]
			\vV\!\left(\tfrac{x}{t}; \oline\vU, \vU_R \right), & x>0.
		\end{cases}
	\end{equation*}
	However, this solution does not satisfy the monotonicity criterion, since $\lambda_1(\vU_-)>0$ while $\lambda_1(\vU_+)<0$. 
	Note that in the special case $k=0$, the corresponding Riemann data yield the classical RP solution, which contains a stationary shock wave. 
	Although this stationary shock is physically-relevant, it does not comply with the monotonicity criterion. 
	Hence, the violation originates from the interaction between a stationary shock and the coupling interface. 
	Consequently, there exists a point in the solution such that $\lambda_1(\vU)=0$. 
	This phenomenon is often referred to as \emph{nonlinear resonance} and may give rise to rich coupling interface; see, e.g., \cite{liu1987nonlinear,isaacson1992nonlinear}.
\end{Remark}

\section{Numerical Results}\label{sec7}
We verify the coupled Rimeann solutions by means of finite volume methods.
The one-dimensonal computational domain $\left[a,b \right] (a<0,b>0)$ is partitioned into $N$ non-overlapping cells with interfaces $a=x_{\frac{1}{2}}<x_{\frac{3}{2}}<...<x_{N+\frac{1}{2}}=b$. 
Each cell is denoted by $I_j=[x_{j-\frac{1}{2}},x_{j+\frac{1}{2}}]$, with center $x_j=(x_{j-\frac{1}{2}}+x_{j+\frac{1}{2}})/2$.
We restrict to a uniform mesh, with mesh size $h=x_{j+\frac{1}{2}}-x_{j-\frac{1}{2}}$.

A natural numerical strategy for the coupled problem is to solve the equations separately on the two subdomains, while enforcing the coupling condition as a boundary condition; see, e.g., \cite{herty2019couplingof,kroner2005numerical}. 
However, the boundary states are in general two-branched. Although one branch can be selected according to the admissibility criterion, our present goal is precisely to verify both the admissibility criterion and the admissible Riemann solutions. Hence, we prefer to avoid introducing any bias stemming from a priori selection.

To this end, we employ the splitting method used in \cite{yu2023well}. Since the system is not conservative on the entire domain $\Omega$, the essential idea is to reformulate the coupling condition as a source-type term within the numerical scheme.

Let $\Delta t>0$ be a uniform time step, and define the discrete time levels by $t^n=n\Delta t$ for $n\in\mathbb{N}$. Denote by $\vU_j^n$ the numerical approximation of the cell average of the exact solution $\vU(x,t)$ over cell $I_j$ at time $t^n$, namely,
\begin{equation*}
	\vU_j^n\approx \frac{1}{h}\int_{I_j}\vU(x,t^n)\mathrm{d}x.
\end{equation*}
In particular, the initial data are discretized as
\begin{equation*}
	\vU_j^n=\frac{1}{h}\int_{I_j}\vU_0(x)\mathrm{d}x.
\end{equation*}

We evolve the numerical solution in two steps.
\begin{itemize}
	\item Step 1 (global update).\\
The flow field is advanced by a standard conservative finite volume scheme,
\begin{equation*}
		\tvU^{n+1}_j=\vU_j^n-\frac{\Delta t}{h}\left[ \int_{t^n}^{t^{n+1}}\vF(\vU(x_{j+\frac{1}{2}},t))dt
		-\int_{t^n}^{t^{n+1}}\vF(\vU(x_{j-\frac{1}{2}},t))dt
		\right] .
\end{equation*}
The flux integrals are approximated by the local Lax–Friedrichs numerical flux.
The time step $\Delta t$ is chosen according to a suitable Courant–Friedrichs–Lewy (CFL) condition to ensure stability.

\item Step 2 (interface correction).\\ 
Since the flow is discontinuous at the coupling interface, we place the interface at a cell boundary,
\begin{equation*}
	a<x_{j_0+\frac{1}{2}}=0<b,\quad 0\leq j_0\leq N.
\end{equation*}
The coupling condition is enforced by a source-type correction in upwind mode
\begin{equation*}\left\lbrace 
	\begin{aligned}
		\vU_{j}^{n+1}&=\tvU_{j}^{n+1}, \quad j\ne j_0+1\\
		\vU_{j}^{n+1}&=\tvU_{j}^{n+1}+\frac{\Delta t}{h}\vS^{n+1},, \quad j= j_0+1,
	\end{aligned}\right. 
\end{equation*}
where the source term $\vS^{,n+1}$ is defined by
\begin{equation*}
		\vS^{n+1}=\begin{pmatrix}
			0\\0\\
			(\widetilde{E}_{j_0+1}^{n+1}+\widetilde{p}_{j_0+1}^{n+1})\widetilde{u}_{j_0+1}^{n+1}-(\widetilde{E}_{j_0}^{n+1}+\widetilde{p}_{j_0}^{n+1})\widetilde{u}_{j_0}^{n+1}
		\end{pmatrix}.
\end{equation*}
\end{itemize}

\subsection{Numerical verification of the three structures (Theorem~\ref{E9})}
We consider three Riemann problem with the following intial conditions
\begin{alignat*}{4}
	\textit{Structure 1: }&\rho_L=1.0,& &u_L=0.8,& &p_L=1.0,&\\ 
	&\rho_R=1.1743,& \quad &u_R=0.5437,& \quad &p_R=0.6203.&\\
	\textit{Structure 2: }&\rho_L=0.8529,& &u_L=1.0046,& &p_L=0.9712,&\\ 
	&\rho_R=0.5028,& \quad &u_R=1.6840,& \quad &p_R=0.2318.&\\
	\textit{Structure 3: }&\rho_L=1.0,& &u_L=2.2,& &p_L=1.0,&\\ 
	&\rho_R=2.3757,& \quad &u_R=0.8766,& \quad &p_R=1.8482.&\\
\end{alignat*}
The corresponding numerical results are displayed in Figures~\ref{fig2}, \ref{fig3}, and \ref{fig4}, respectively.
The reference (exact) solutions are constructed by sequentially resolving each simple wave.
In all three test cases, the numerical solutions are found to be in excellent agreement with the exact ones, thereby validating the proposed Riemann solutions.
\begin{figure}[!h]
	\centering
	\subfigure[density]{
		\includegraphics[width=0.31\linewidth]{ 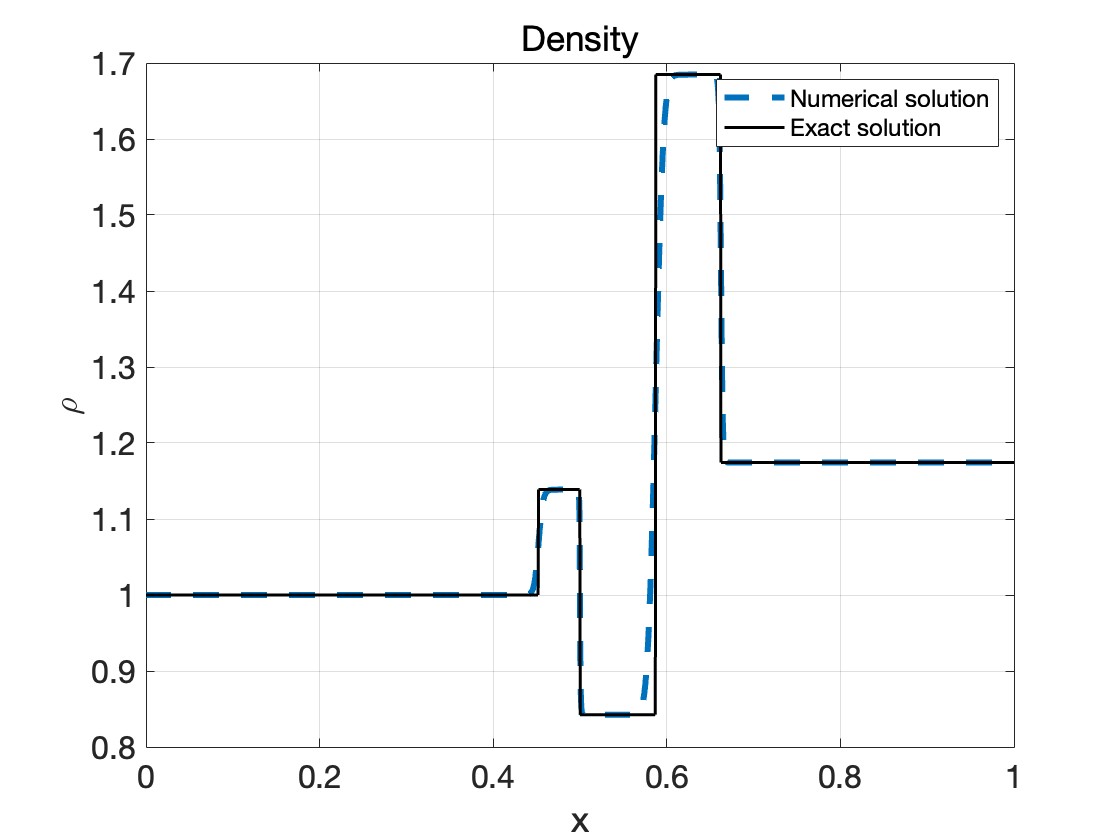}
	}
	\subfigure[velocity]{
		\includegraphics[width=0.31\linewidth]{ 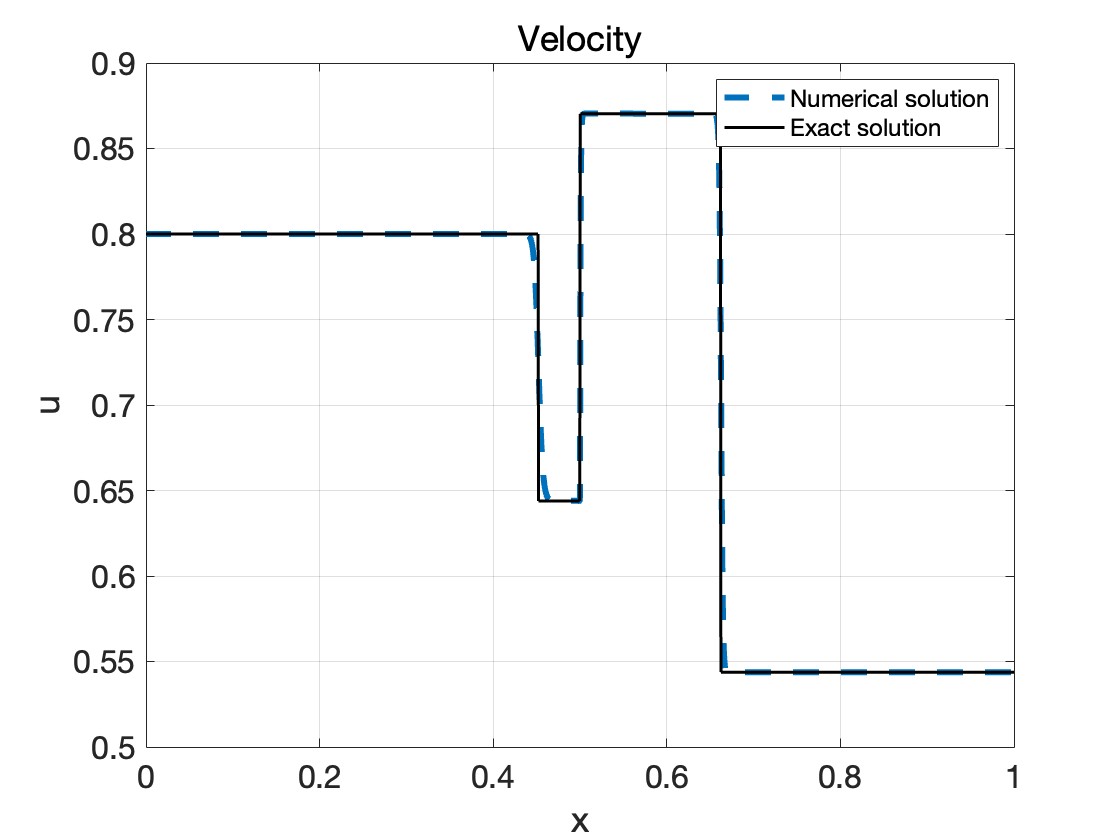}
	}
	\subfigure[pressure]{
		\includegraphics[width=0.31\linewidth]{ 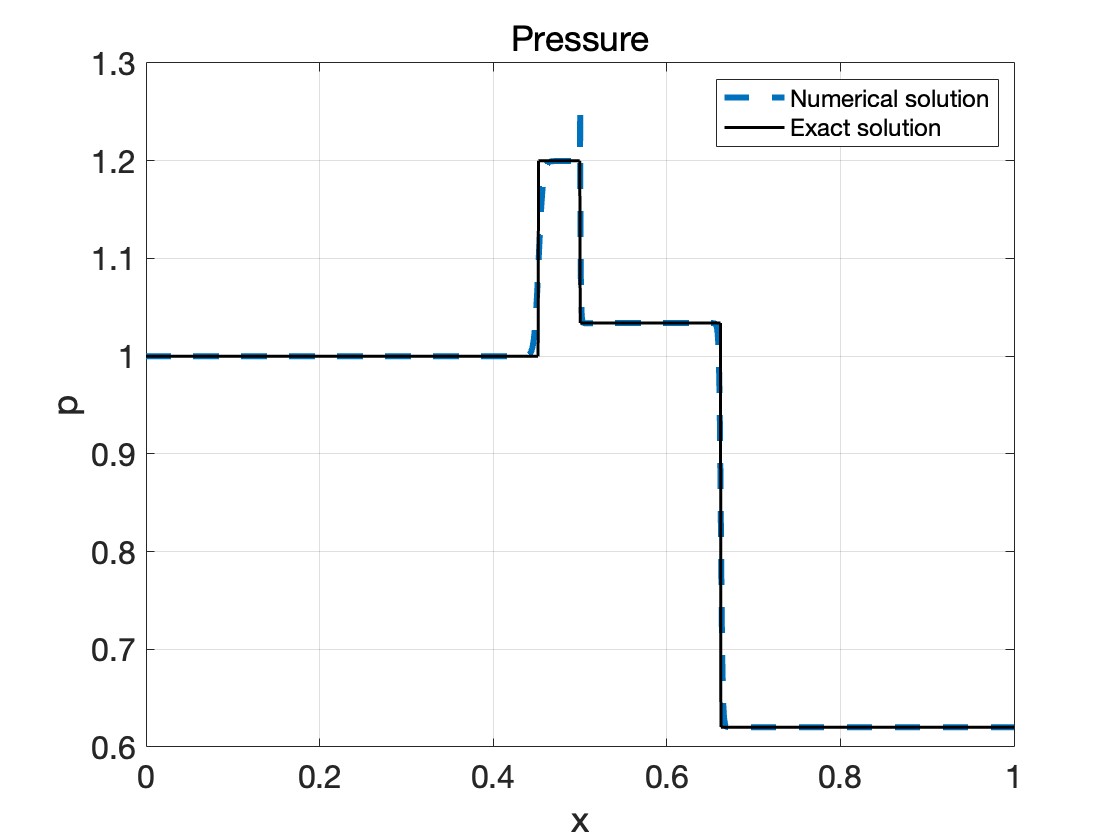}
	}
	\caption{Comparison of numerical and exact solutions for the Riemann problem of \textit{Structure 1}.}
	\label{fig2}
\end{figure}
\begin{figure}[!h]
	\centering
	\subfigure[density]{
		\includegraphics[width=0.31\linewidth]{ 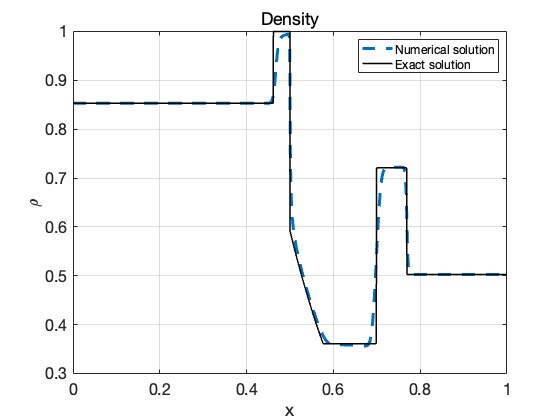}
	}
	\subfigure[velocity]{
		\includegraphics[width=0.31\linewidth]{ 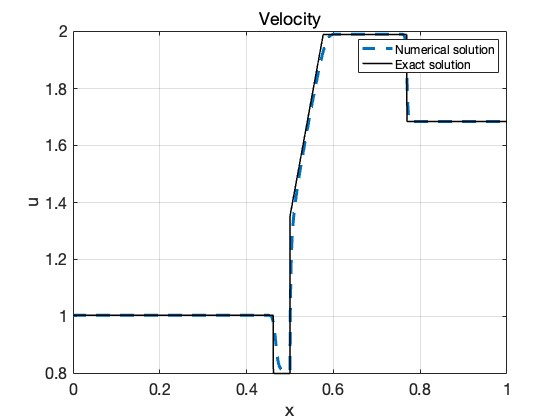}
	}
	\subfigure[pressure]{
		\includegraphics[width=0.31\linewidth]{ 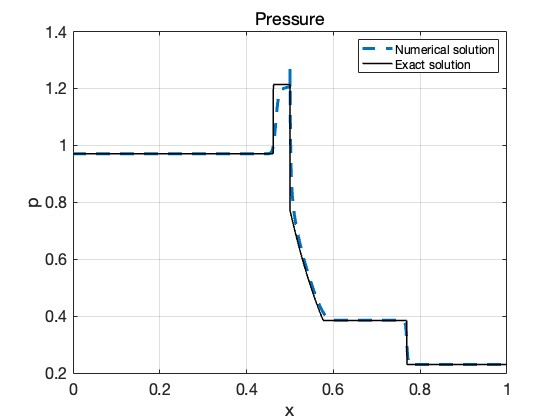}
	}
	\caption{Comparison of numerical and exact solutions for the Riemann problem of \textit{Structure 2}.}
	\label{fig3}
\end{figure}
\begin{figure}[!h]
	\centering
	\subfigure[density]{
		\includegraphics[width=0.31\linewidth]{ 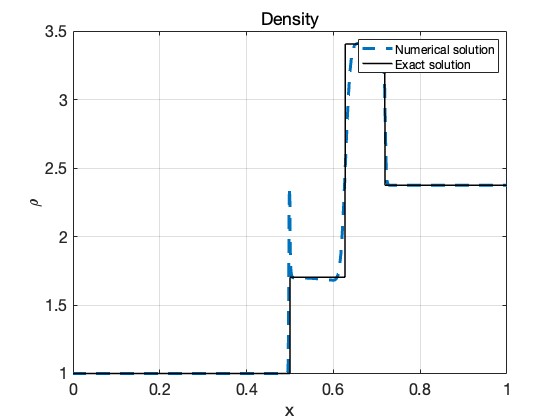}
	}
	\subfigure[velocity]{
		\includegraphics[width=0.31\linewidth]{ 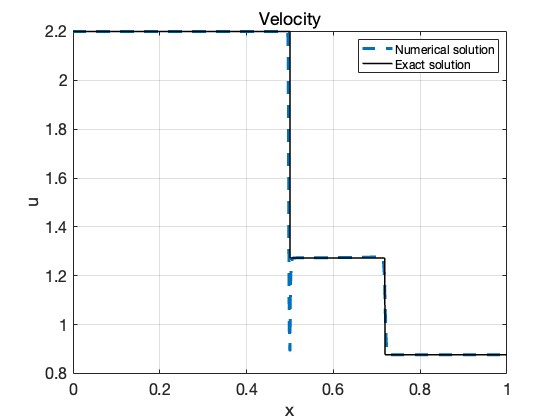}
	}
	\subfigure[pressure]{
		\includegraphics[width=0.31\linewidth]{ 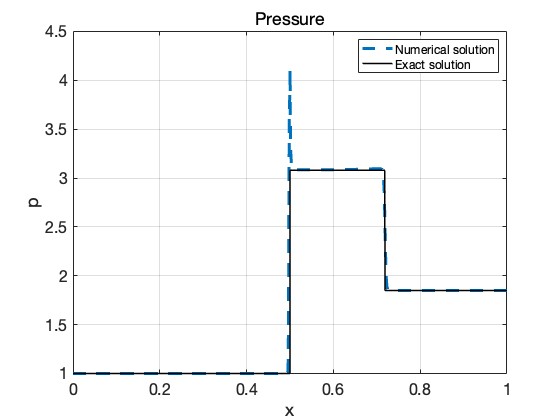}
	}
	\caption{Comparison of numerical and exact solutions for the Riemann problem of \textit{Structure 3}.}
	\label{fig4}
\end{figure}

\subsection{Numerical verification of the physically admissible criterion (Remark~\ref{R1})}

According to the physical admissibility criterion, the coupled Riemann solutions are expected to converge to the classical Riemann solution as $k\to 0$. To verify this property, we compute numerical solutions for a range of $k$ values. The initial condition is chosen to be the classical Sod shock-tube problem,
\begin{alignat*}{4} 
	&\rho_L=1.0,& &u_L=0.0,& &p_L=1.0,&\\ &\rho_R=0.125,& \quad &u_R=0.0,& \quad &p_R=0.1.&
\end{alignat*}
Figure~\ref{fig5} presents the numerical results obtained with five different parameter values $k=0,0.05,0.1,0.2,0.3$. The results clearly demonstrate that, as $k$ decreases to zero, the numerical solutions approach the classical Riemann solution corresponding to $k=0$, in full agreement with the physical admissibility criterion. 
\begin{figure}[!h]
	\centering
	\subfigure[density]{
		\includegraphics[width=0.31\linewidth]{ 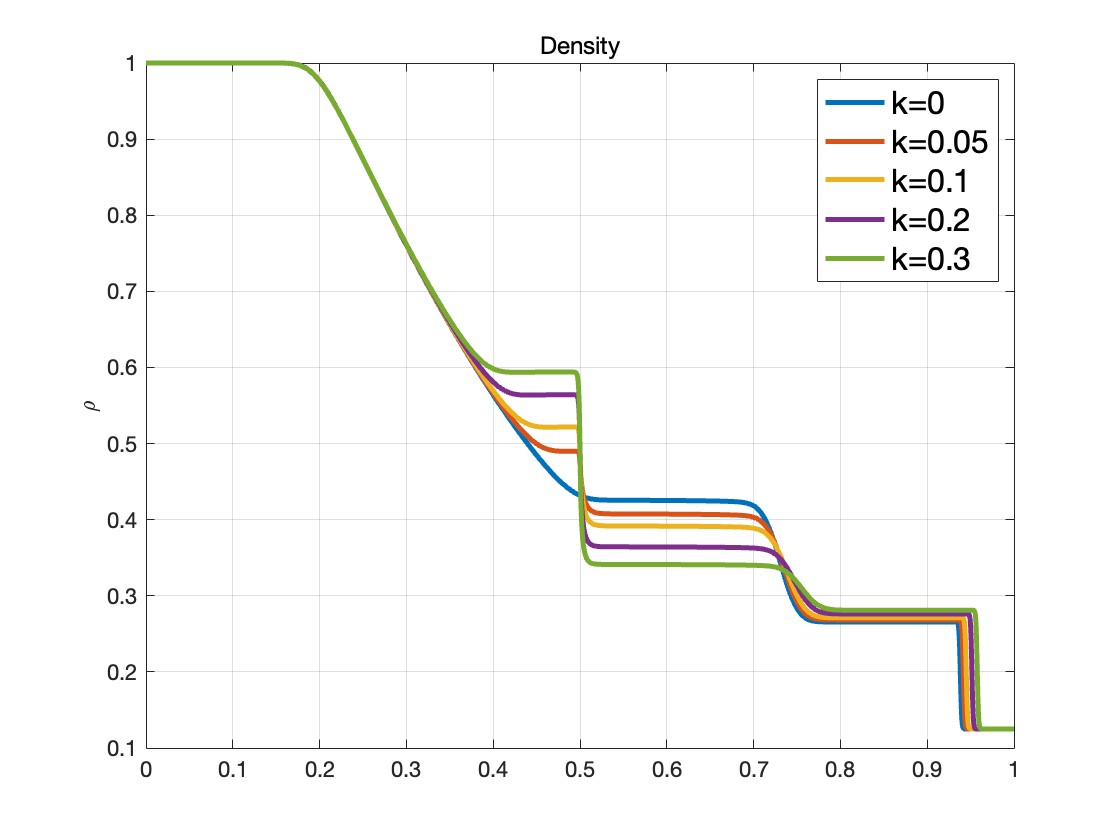}
	}
	\subfigure[velocity]{
		\includegraphics[width=0.31\linewidth]{ 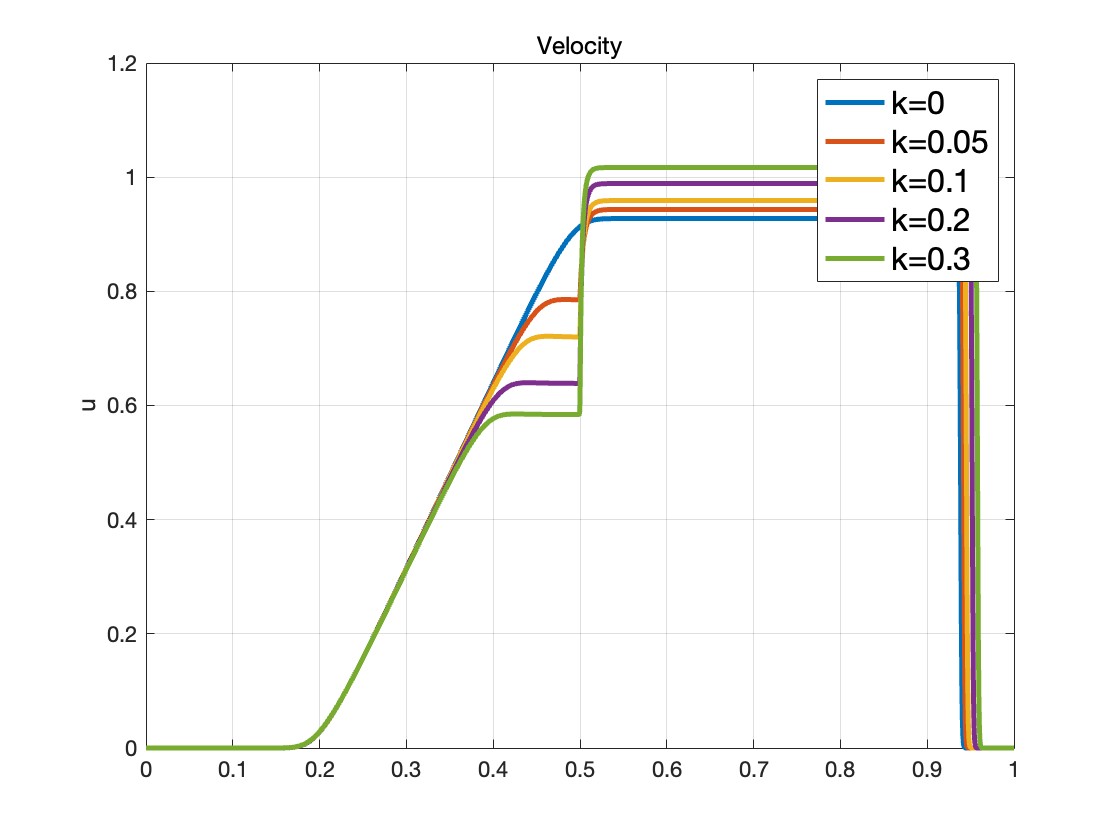}
	}
	\subfigure[pressure]{
		\includegraphics[width=0.31\linewidth]{ 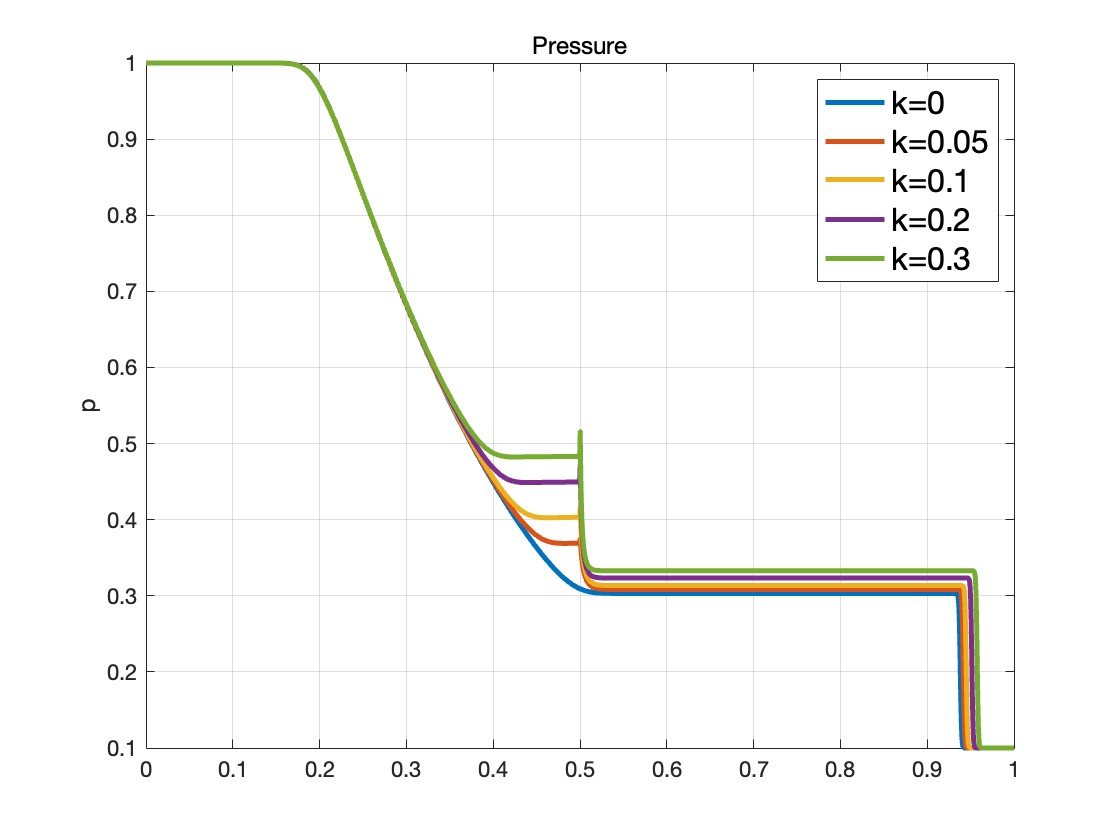}
	}
	\caption{The numerical results of the Riemann problem with different $k$.}
	\label{fig5}
\end{figure}

\subsection{A test of the Riemann problem in the absence of admissible solutions}

We construct the problem described in Theorem~\ref{E8} in such a way that no admissible Riemann solution exists. The initial condition is specified as
\begin{alignat*}{4}
	&\rho_L=1.0, & \quad &u_L=2.2, & \quad &p_L=1.0, &\\
	&\rho_R=2.8610, & \quad &u_R=1.0719, & \quad &p_R=2.0891. &
\end{alignat*}
Figure~\ref{fig6} shows the numerical results for the density, pressure, and Mach number. The computed solution remains self-similar and consists of three waves: a stationary discontinuity at the origin, a right-moving contact discontinuity, and a right-moving shock wave. The Mach number profile indicates that the left trace at the coupling interface is supersonic, whereas the right trace is subsonic. This implies that the numerical solution violates the monotonicity criterion.

\begin{figure}[!h]
	\centering
	\subfigure[density]{
		\includegraphics[width=0.31\linewidth]{ 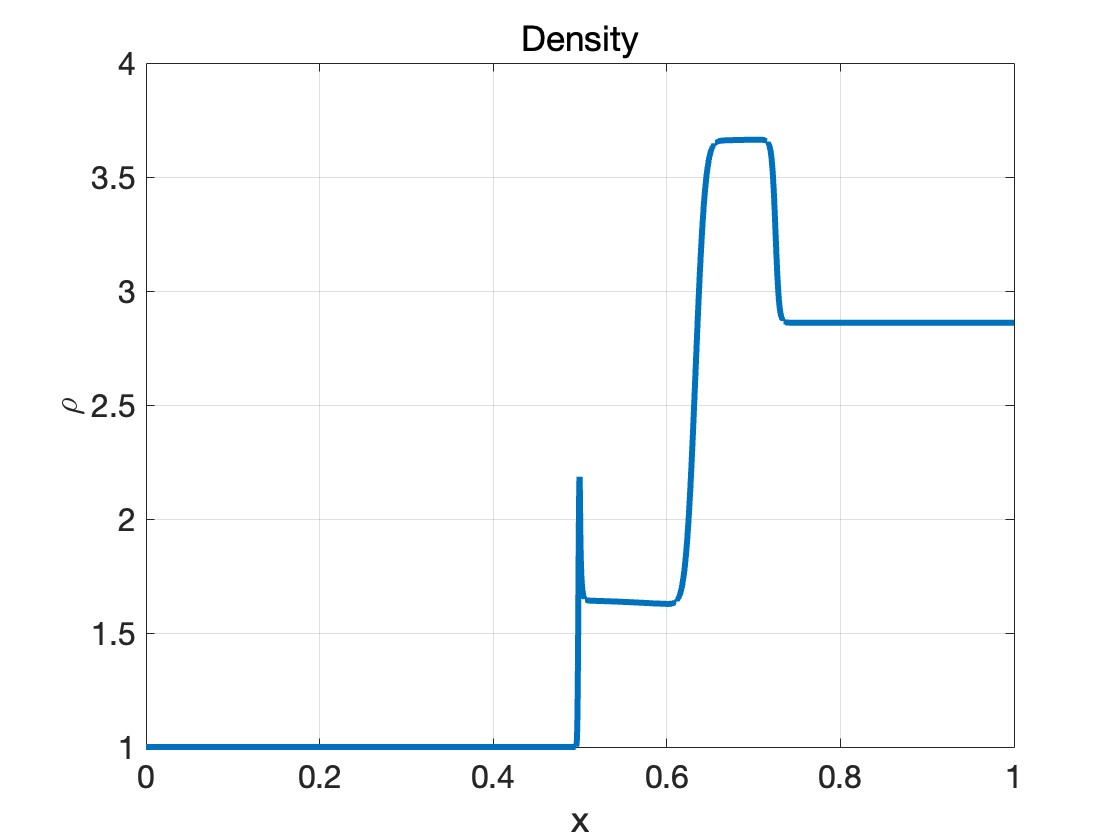}
	}
	\subfigure[pressure]{
		\includegraphics[width=0.31\linewidth]{ 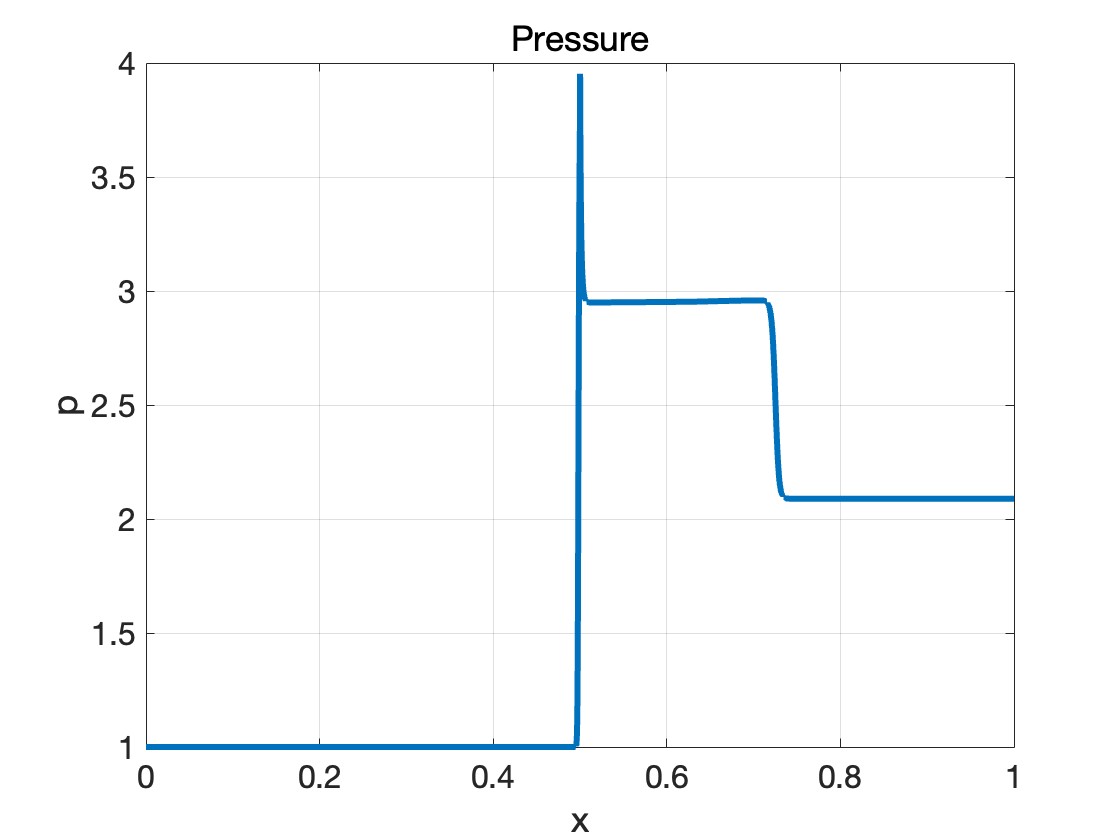}
	}
	\subfigure[Mach number]{
		\includegraphics[width=0.31\linewidth]{ 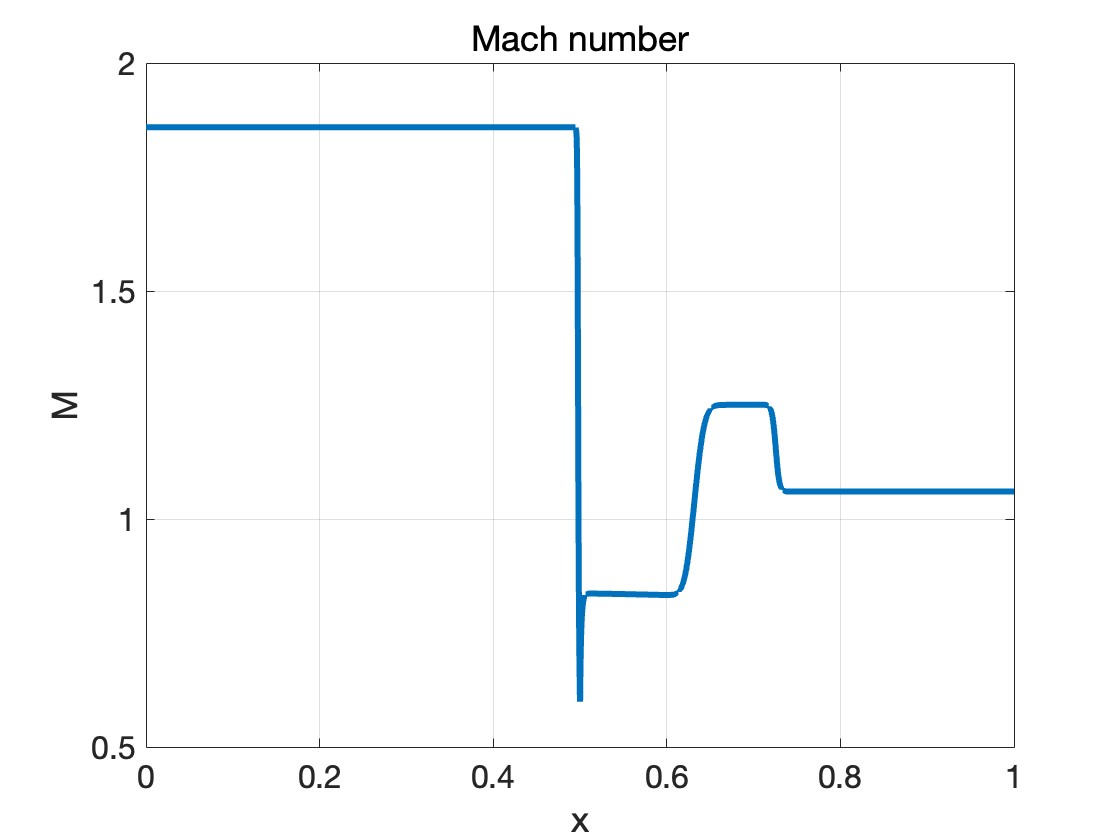}
	}
	\caption{The numerical results of the Riemann problem without an admissible solution.}
	\label{fig6}
\end{figure}

\section{Conclusion}\label{sec8}
In this work, we have studied the coupled Riemann problem with a heat flux discontinuity by the half-Riemann problem approach. Owing to the different regimes of the Mach number, several solution structures may occur. We have shown that weak solutions are not necessarily unique and introduced an admissibility criterion to select the physically relevant ones. The local existence of admissible solutions has been established. Furthermore, we have constructed a family of initial data for which no admissible solution exists. Finally, several theoretical findings have been verified by numerical experiments.


\section*{Conflicts of interest. }
On behalf of all authors, the corresponding author states that there is no conflict of interest.

\appendix

\section{Lax curves of the Euler system}\label{App2}

The classical Riemann solution of the Euler system can be constructed from simple waves associated with the three characteristic fields. In this section, we list the Lax curves of these waves. The 2-characteristic field is linearly degenerate, while the 1- and 3-characteristic fields are genuinely nonlinear. For simplicity, we only explicitly list the Lax curves for the 1-field and 2-field; the 3-field can be obtained analogously.  

\subsection{1-characteristic field}

The Lax curve \(L_1^{\pm}(\eps;\vU_0)\) is expressed in terms of the primitive variables: density \(\rho^{\pm}(\eps;\vU_0)\), velocity \(u^{\pm}(\eps;\vU_0)\), and pressure \(p^{\pm}(\eps;\vU_0)\). We take the parameter \(\eps\) as
\[
\eps = \frac{p^{\pm}(\eps;\vU_0)}{p_0}-1,
\]
so that 
\[
p^{\pm}(\eps;\vU_0) = p_0(\eps+1).
\]

\textbf{Shock wave:} If \(\eps \ge 0\), then the 1-wave is a shock, and
\begin{equation*}
	\begin{aligned}
		\rho^{\pm}(\eps;\vU_0)&=\rho_0\frac{(\gamma+1)(\eps+1)+(\gamma-1)}{(\gamma-1)(\eps+1)+(\gamma+1)},\\
		u^{\pm}(\eps;\vU_0)&=u_0+c_0(\rho_0-\rho^{\pm}(\eps;\vU_0))\big(\rho_0+\frac{\gamma-1}{2}(\rho_0-\rho^{\pm}(\eps;\vU_0))\big)
	\end{aligned}
\end{equation*}

\textbf{Rarefaction wave:} If \(\eps < 0\), then the 1-wave is a rarefaction, and
\[
\begin{aligned}
	\rho^{\pm}(\eps;\vU_0) &= \rho_0 (1+\eps)^{1/\gamma},\\
	u^{\pm}(\eps;\vU_0) &= u_0 + \frac{2c_0}{\gamma-1} \left[ 1 - (1+\eps)^{(\gamma-1)/(2\gamma)} \right].
\end{aligned}
\]

It is easy to check that \(L_1^{\pm}(\eps;\vU_0)\) is twice continuously differentiable with respect to \(\eps\), and
\[
\frac{d u^{\pm}(\eps;\vU_0)}{d\eps} < 0, \quad \frac{d^2 u^{\pm}(\eps;\vU_0)}{d\eps^2} > 0,
\]
see, e.g., \cite{godlewski2013numerical}.  

\subsection{2-characteristic field}

Across the contact discontinuity (2-field), the pressure and velocity remain constant. The Lax curve \(L_2^{\pm}(\eps;\vU_0)\) can be written as
\[
\rho^{\pm}(\eps;\vU_0) = \rho_0 (1+\eps), \quad u^{\pm}(\eps;\vU_0) \equiv u_0, \quad p^{\pm}(\eps;\vU_0) \equiv p_0.
\]

\section{Proof of Proposition~\ref{P4}}\label{App1}

We first consider $h_1(M,k)$. Recall that
\begin{align*}
	h_1(M,k)
	&=\frac{\gamma M^2+1}{\gamma (\psi(M))^2+1}
	=\frac{\gamma M^2+1}{\gamma \frac{1- I}{1+\gamma I}+1}
	=\frac{(\gamma M^2+1)(1+\gamma I)}{\gamma+1} \\
	&=\frac{1}{\gamma+1}\left( \gamma M^2+1+\gamma \left[ (\gamma M^2+1)^2-(\gamma+1)M^2\big((\gamma-1)M^2+2\big)(1+k)\right]^{1/2}\right).
\end{align*}
Denote $t=M^2<1$. Then
\begin{equation*}
	\frac{\gamma+1}{2M}\frac{\partial h_1}{\partial M}
	= \frac{-\gamma\big(k-t+\gamma k-kt+\gamma^2kt+1\big)}
	{\left[ (\gamma t+1)^2-(\gamma+1)t\big((\gamma-1)t+2\big)(1+k)\right]^{1/2}}+\gamma.
\end{equation*}
Since
\begin{align*}
	&\big(k-t+\gamma k-kt+\gamma^2kt+1\big)^2-\Big[(\gamma t+1)^2-(\gamma+1)t\big((\gamma-1)t+2\big)(1+k)\Big] \\
	&=\big[(1-t)+k(1-t)+\gamma k(\gamma t+1)\big]^2-(1-t)^2+k(\gamma t+1)^2-k(1-t)^2 \\
	&=k(k+1)(1-t)^2+\gamma^2 k^2 (\gamma t+1)^2+2(1-t)\gamma k (\gamma t+1)+2\gamma k^2(1-t)(\gamma t+1)+k(\gamma t+1)^2 \\
	&>0,
\end{align*}
we deduce
\begin{equation*}
	\frac{\partial h_1}{\partial M}(M,k)<0, \qquad \text{for any } M<1,\,k>0.
\end{equation*}

We then consider $h_2(M,k)$. By definition,
\begin{align*}
	h_2(M,k)
	&=\frac{(\psi(M,k))^2(\gamma M^2+1)}{M^2(\gamma (\psi(M,k))^2+1)}
	=\frac{(\gamma M^2+1)(1- I)}{(\gamma+1)M^2} \\
	&=\frac{\gamma M^2+1-\big[(\gamma M^2+1)^2-(\gamma+1)M^2\big((\gamma-1)M^2+2\big)(1+k)\big]^{1/2}}{(\gamma+1)M^2}.
\end{align*}
Its derivative is
\begin{equation*}
	\frac{\gamma+1}{2M}\frac{\partial h_2}{\partial M}=
	\frac{1}{(\gamma+1)t^2}\left\{ 
	\frac{1-kt-k\gamma t-t}{\left[(\gamma t+1)^2-t(\gamma+1)\big((\gamma-1)t+2\big)(1+k)\right]^{1/2}}-1
	\right\}.
\end{equation*}
When $k<\tfrac{1}{\gamma^2-1}$, we have
\begin{equation*}
	1-kt-k\gamma t-t>0.
\end{equation*}
Moreover,
\begin{align*}
	&(1-kt-k\gamma t-t)^2-\Big[(\gamma t+1)^2-t(\gamma+1)\big((\gamma-1)t+2\big)(1+k)\Big] \\
	&=k(1+k)(\gamma+1)^2t^2>0.
\end{align*}
Hence
\begin{equation*}
	\frac{\partial h_2}{\partial M}(M,k)>0, \qquad \text{for any } M<1,\,k>0.
\end{equation*}
\qed

\bibliography{reference}
\bibliographystyle{plain}
\end{document}